\documentclass[12pt,twoside,reqno]{amsart}
\usepackage{amsmath}
\usepackage{amssymb}
\usepackage{wasysym}
\usepackage{psfrag}
\usepackage{graphicx,epsf,amsmath}  
\usepackage{epsf,graphicx}
\usepackage{comment}

\usepackage{hyperref}

\newtheorem{lem}{Lemma}[section]
\newtheorem{theorem}[lem]{Theorem}
\newtheorem{propo}[lem]{Proposition}
\newtheorem{defi}[lem]{Definition}
\newtheorem{coro}[lem]{Corollary}
\newtheorem{rema}[lem]{Remark}

\numberwithin{equation}{section}

\def\RR{{\mathbb R}}

\def\CC{{\mathbb C}}

\def\dfrac{\displaystyle\frac}

\def\R{\mathbb{R}}
\def\epsilon{\varepsilon}

\newcommand{\abs}[1]{\lvert#1\rvert}

\title[Layer solutions for fractional Laplacians]{Nonlinear equations for fractional Laplacians II: 
existence, uniqueness, and qualitative properties of solutions}
\author{Xavier Cabr\'e}
\author{Yannick Sire}

\thanks{The first author was supported by
grants MTM2008-06349-C03-01,  MTM2011-27739-C04-01 (Spain) and 
2009SGR-345 (Catalunya). The second author is supported by the ANR project PREFERED}
\address{X. Cabr\'e: ICREA and Universitat Polit{\`e}cnica de Catalunya,
Departament de Matem{\`a}tica Aplicada I, Diagonal 647, 08028
Barcelona, Spain}
\email{xavier.cabre@upc.edu}
\address{Y. Sire: Universit\'e Paul C\'ezanne, LATP,
Facult\'e des Sciences et Techniques, Case cour A,
Avenue Escadrille Normandie-Niemen, F-13397 Marseille Cedex 20,
France, and CNRS, LATP, CMI, 39 rue F. Joliot-Curie, F-13453 Marseille Cedex
13, France}
\email{sire@cmi.univ-mrs.fr}

\begin{document}

\begin{abstract}
This paper, which is the follow-up to part~I, concerns the equation
$(-\Delta)^{s} v+G'(v)=0$ in $\mathbb{R}^{n}$,
with $s\in (0,1)$, where $(-\Delta)^{s}$ stands for the fractional
Laplacian ---the infinitesimal generator of a L\'evy process. 

When $n=1$, we prove that there exists a layer solution of the equation 
(i.e., an increasing solution with limits $\pm 1$ at $\pm \infty$)
if and only if the potential $G$ has only two absolute minima
in $[-1,1]$, located at $\pm 1$ and satisfying $G'(-1)=G'(1)=0$. 
Under the additional hypothesis $G''(-1)>0$ and $G''(1)>0$, we also establish its uniqueness and 
asymptotic behavior at infinity. 
Furthermore, we provide with a concrete, almost explicit, example
of layer solution.

For $n\geq 1$, we prove some results related to the 
one-dimensional symmetry of certain solutions ---in the spirit of a well-known
conjecture of De Giorgi for the standard Laplacian.

\end{abstract}

\maketitle

\section{Introduction}

This paper, which is a follow-up to our work \cite{CS1}, is devoted to study the 
nonlinear problem
\begin{equation}
\label{problem}
(-\Delta)^{s} v= f(v)\,\,\,\text{ in } \mathbb{R}^{n},
\end{equation}
where $s \in (0,1)$ and 
\begin{equation}\label{deflapl}
(-\Delta)^{s} v(x)=C_{n,s}\ \text{P.V.}\int_{\mathbb{R}^{n}}
\frac{v(x)-v(\overline x )}{|x-\overline x|^{n+2s}}\,d\overline  x 
\end{equation}
is the fractional Laplacian.  
In the previous integral, P.V. stands for the Cauchy principal value
and $C_{n,s}$ is a normalizing constant 
to guarantee that the symbol of the resulting operator is $|\xi|^{2s}$;
see \cite{CS1} for more details. As explained in
section~3 of \cite{CS1}, this problem is equivalent to the nonlinear 
boundary value problem 
\begin{equation}
\label{extAlpha}
\begin{cases}
\textrm{div\,}(y^{a}\,\nabla u)=0&\text{ in } \mathbb{R}^{n+1}_+\\ 
(1+a)\displaystyle{\frac{\partial u}{\partial{\nu}^{a}}}
=f(u) &\text{ on } \partial\mathbb{R}^{n+1}_+ ,
\end{cases}
\end{equation}
where $n\geq 1$, $\mathbb{R}^{n+1}_+=\{(x,y) \in \mathbb{R}^n \times \mathbb{R} : y>0\}$
is a halfspace, $\partial\mathbb{R}^{n+1}_+=\{y=0\}$,
$u=u(x,y)$ is real valued, and $$\frac{\partial u}{\partial
  {\nu}^{a}}=-\lim_{y \to 0} y^{a} \partial_y u$$
is the generalized exterior normal derivative of~$u$.
Points in $\mathbb{R}^{n}$ are denoted by 
$x=(x_1,\ldots ,x_{n})$. The parameter $a$ belongs to $(-1,1)$ 
and is related to the power of the fractional
Laplacian $(-\Delta)^s$ by $$a=1-2s. $$ Indeed, Caffarelli and Silvestre (see
\cite{cafS, CS1}) proved the following formula relating the fractional
Laplacian $(-\Delta)^s$ to the Dirichlet-to-Neumann operator: 
\begin{equation}\label{cttNeumann}
(-\Delta)^{s} \left \{u(\cdot, 0) \right \}=d_s \frac{\partial u}{\partial
  \nu^{a}} \quad\text{ in } \mathbb{R}^{n}=\partial\mathbb{R}^{n+1}_+, 
\end{equation}
where $d_s$ is a positive constant depending only on $s$. 

The aim of the present paper is to study some special bounded solutions of~\eqref{problem}. 
The solutions we consider are the so-called {\it layer solutions}, i.e., those solutions
which are monotone increasing, 
connecting  $-1$ to $1$ at $\mp\infty$, 
in one of the $x$-variables.  We focus on their existence, uniqueness,
symmetry and variational properties, as well as their asymptotic
behavior. 

In our previous paper \cite{CS1}, we proved a Modica-type estimate which allowed to derive a 
necessary condition on the nonlinearity $f$
for the existence of a layer solution in $\RR$. More precisely, 
we proved the following result.

\begin{theorem}[\cite{CS1}]\label{modthm} 
Let $a \in (-1,1)$ and $f$ any $C^{1,\gamma}(\RR)$ function, 
for some $\gamma >\max(0,a)$. Let $n=1$ and $u$ be a layer solution of \eqref{extAlpha}, 
that is, a bounded solution of \eqref{extAlpha} with $n=1$ such that $u_x(\cdot,0) >0$ 
in $\RR$ and $u(x,0)$ has limits $\pm 1$ as $x \rightarrow \pm \infty$.

Then, for
every $x\in\RR$ we have 
$\int_0^{+\infty}t^{a}|\nabla u(x,t)|^2 dt<\infty$ and the Hamiltonian equality
\begin{equation*}
(1+a)\int_0^{+\infty} \frac{1}{2} t^a \left\{u_x^2(x,t)-u_y^2(x,t)\right\} dt=G(u(x,0))-G(1). 
\end{equation*}
Furthermore, for all $y \geq 0$ and  $x \in \RR$ we have
\begin{equation*}
(1+a) \int_0^y\dfrac{1}{2} t^{a} \left\{
u_x^2(x,t)-u_y^2(x,t)\right\} dt < G(u(x,0))-G(1).
\end{equation*}
 
\end{theorem}

In the previous theorem, the last estimate is uniform as $s$ tends to $1$, i.e.,
as $1+a$ tends to $0$. This led in \cite{CS1} to the convergence of layers, as $s \uparrow 1$, 
to a layer of $-v''=f(v)$ in $\RR$. 
In addition, using the Hamiltonian estimates of Theorem \ref{modthm}, we established 
the following necessary conditions for the existence of a layer in $\RR$. 

\begin{theorem}[\cite{CS1}]\label{necLayers}
Let $s \in (0,1)$ and $f$ any $C^{1,\gamma}(\RR)$ function, 
for some $\gamma >\max(0,1-2s)$. Assume that there exists a layer solution $v$ of  
\begin{equation}\label{eqcc}
(-\partial_{xx})^s v= f(v)\,\,\,\,\text{in }\RR ,
\end{equation}
that is, $v$ is a solution of \eqref{eqcc} satisfying 
$$v'>0 \quad \text{in }\RR \qquad\text{and}\qquad 
\lim_{x \rightarrow \pm \infty} v(x)=\pm 1. $$

Then, we have  
\begin{equation}\label{nec1}
G'(-1)=G'(1)=0 
\end{equation}
and 
\begin{equation}\label{nec2}
G>G(-1)=G(1)\,\,\,\,\text{in }(-1,1).
\end{equation}
\end{theorem}

In the present paper, we prove that the two necessary conditions in Theorem \ref{necLayers} are 
actually sufficient to ensure the existence of a layer solution in $\RR$. 
Under the additional hypothesis $G''(-1) >0$ and $G''(1) >0$, we also prove 
the uniqueness (up to
translations) of a layer solution in $\RR$ and we establish its asymptotic behavior at infinity. 

To study the asymptotic behavior of the layer solution for a given nonlinearity $f$,  
it will be very useful to have the following almost explicit example 
of layer solution for a particular nonlinearity. 
For every $t>0$, a layer solution for some odd nonlinearity
$f^t_s\in C^1([-1,1])$ (see Theorem~\ref{examen} below for more details)
is provided by the following function: 
\begin{equation}\label{explicit}
v^t_s(x)=-1+2\int_{-\infty}^x p_s(t,\overline{x})\,d\overline{x}=
\frac{2}{ \pi}\int_0^\infty\frac{\sin(xr)}{r}e^{-tr^{2s}} dr,
\end{equation}
where $p_s$ is the fundamental solution of the linear fractional heat equation
\begin{equation*}
\partial_t w+(-\partial_{xx})^s w=0, \qquad t>0,\,\,\,x \in \RR.
\end{equation*}

When $s=1/2$, the particular layer solution above agrees with the explicit one used in 
\cite{CSM}, namely
$$
v^t_{1/2}(x)=\frac{2}{\pi}\arctan\frac{x}{t}, \quad\text{with }
\,\, f^t_{1/2}(v)=\frac{1}{\pi t}\sin (\pi v).
$$
In \cite{CSM}, J. Sol\`a-Morales and one of the authors studied layer solutions of 
\begin{equation*}
\begin{cases}
\Delta u=0&\text{ in } \mathbb{R}^{n+1}_+\\ 
\displaystyle{\frac{\partial u}{\partial{\nu}}}
=f(u) &\text{ on } \partial\mathbb{R}^{n+1}_+ ,
\end{cases}
\end{equation*}
which corresponds to the case $a=0$ (that is $s=1/2$) in \eqref{extAlpha}. 
The goal of our paper is to generalize this study to any fractional
power of the Laplacian between $0$ and $1$. We will make a great use of the tools developed in 
\cite{CSM}.

The study of elliptic equations involving fractional powers of the
Laplacian appears to be important in many physical situations in which one has to consider long-range or
anomalous diffusions. From a probabilistic point of view, the
fractional Laplacian appears as the infinitesimal generator of  a
L\'evy process (see the book of Bertoin \cite{bertoin}). In our case, as
in \cite{CSM}, we will concentrate on the problem \eqref{extAlpha} and
we will not consider probabilistic aspects.

Problem \eqref{extAlpha} is clearly a degenerate elliptic problem concerning the
weight $y^{a}$. However, since $a \in (-1,1)$, the weight
$y^{a}$ belongs to the Muckenhoupt class of $A_2$ functions, i.e., it satisfies 
\begin{equation*}
\sup_B \, (\frac{1}{|B|} \int_B w)(\frac{1}{|B|} \int_B w^{-1}) \leq C
\end{equation*}
where $w(x,y)=|y|^{a}$ and $B$ denotes any ball in $\mathbb{R}^{n+1}$. 
This fact allows to develop a regularity theory for weak solutions of \eqref{extAlpha};
see \cite{CS1}. 

Another important property of the weight $y^{a}$ is that it just
depends on the extension variable $y$ and not on the tangential variable
$x$. The equation is therefore invariant under translations in $x$,
which allows the use of the sliding method to get uniqueness of layer solution in $\RR$,
as well as monotonicity of solutions with limits $\pm 1$ at $\pm \infty$.

\begin{rema}
{\rm
Another interesting problem is to consider the existence of monotone solutions
of equation \eqref{problem} connecting $\underline{v}(x_2,...,x_{n})$ at
$-\infty$ to $\overline{v}(x_2,...,x_{n})$ at $+\infty$ where both
$\underline{v}$ and $\overline{v}$ are solutions of
$(-\Delta)^sw=f(w)$ in $\RR^{n-1}$. We will not address this problem here,
but we believe that
the methods developed in the present paper (and in \cite{CSM,CS1})
allow to deal with this type of problem.      
}
\end{rema}

\section{Results}

Throughout the paper we will assume that the nonlinearity $f$ is of class
$C^{1,\gamma}(\RR)$ for some $\gamma > \max(0,1-2s)$. We will denote by $G$ the
associated potential, i.e., 
\begin{equation*}
G'=-f. 
\end{equation*}
The potential $G$ is uniquely defined up to an additive constant. 

Let $P_s=P_s(x,y)$ be the Poisson kernel 
associated to the operator $L_a=\textrm{div\,}(y^{a}\,\nabla)$,
with $a=1-2s$. 
We then have (see section~3 of \cite{CS1}):
for $v$ a bounded $C^2_\textrm{loc}(\RR^n)$ function, $v$ 
is a solution of \eqref{problem} if and only if 
$$u(\cdot,y)=P_s(\cdot,y) * v,$$
a function having $v$ as trace on $\partial \RR^{n+1}_+$,
is a solution of \eqref{extAlpha} with $f$ replaced by $(1+a) d_s^{-1} f
= 2(1-s)d_s^{-1} f$.
Recall that $d_s$ is the constant from \eqref{cttNeumann}.  
It turns out that $2(1-s)d_s^{-1}$ has a positive limit as $s\uparrow 1$.
This is the reason why we wrote problem \eqref{extAlpha}  in
\cite{CS1} with the multiplicative constant 
$1+a=2(1-s)$ in it; we wanted uniform estimates as  $s\uparrow 1$.

Let us recall some regularity results from \cite{CS1}. The first one is Lemma 4.4 of 
\cite{CS1}.

\begin{lem}[\cite{CS1}]\label{regNL}
Let $f$ be a $C^{1,\gamma}(\RR)$ function with $\gamma >\max(0,1-2s)$. 
Then, any bounded solution of 
$$(-\Delta)^s v =f(v)\,\,\,\text{in }\RR^n$$
is $C^{2,\beta}(\RR^n)$ for some $0 <\beta < 1$ depending only on $s$ and $\gamma$.

Furthermore, given $s_0>1/2$ there exists $0 <\beta < 1$ depending only on $n$,
$s_0$, and $\gamma$ ---and hence independent of $s$--- such that for every $s>s_0$, 
$$\|v\|_{C^{2,\beta}(\RR^n)} \leq C $$
for some constant $C$ depending only on $n$, $s_0$, $\|f\|_{C^{1,\gamma}}$,
and $\|v\|_{L^{\infty}(\RR^n)}$ ---and hence independent of $s \in (s_0,1)$.

In addition, the function defined by $u(\cdot, y)=P_s(\cdot,y) \, * \, v$ 
(where $P_s$ is the Poisson kernel 
associated to the operator $L_a$) satisfies for every $s>s_0$,
$$\|u\|_{C^{\beta}(\overline{\RR^{n+1}_+})} + \|\nabla_x u\|_{C^{\beta}(\overline{\RR^{n+1}_+})}
+\|D^2_x u\|_{C^{\beta}(\overline{\RR^{n+1}_+})} \leq C $$
for some constant $C$ independent of $s\in (s_0,1)$, indeed 
depending only on the same quantities as the previous one.
\end{lem}

Following \cite{CSM},
we introduce
\begin{align*}
& B_R^+=\{ (x,y)\in\R^{n+1} : y>0, |(x,y)|<R\}, \\
& \Gamma_R^0=\{ (x,0)\in\partial\R^{n+1}_+ : |x|<R\}, \\
& \Gamma_R^+=\{ (x,y)\in\R^{n+1} : y\ge 0, |(x,y)|=R\}. 
\end{align*}
We consider the problem in a half-ball
\begin{equation}\label{temp}
\begin{cases}
\textrm{div\,}(y^{a}\,\nabla u)=0&\text{ in } B_R^+\\ 
(1+a)\displaystyle{\frac{\partial u}{\partial{\nu}^{a}}}
=f(u) &\text{ on } \Gamma_R^0 .
\end{cases}
\end{equation} 
In the sequel we will denote by 
$$L_{a}=\textrm{div\,}(y^{a}\,\nabla) $$ 
the differential operator in \eqref{temp}. Obviously, there is a natural 
notion of weak solution of \eqref{temp}; see \cite{CS1}.

We have the following regularity result (Lemma~4.5 of \cite{CS1}).
\begin{lem}[\cite{CS1}]
\label{regularity1} 
Let $a \in (-1,1)$ and $R>0$. Let $\varphi \in C^\sigma (\Gamma^0_{2R})$ for some $\sigma \in (0,1)$ and 
$u \in L^\infty(B^+_{2R}) \cap H^1(B^+_{2R},y^a)$ be a weak solution of
\begin{equation*}
\label{problemBR}
\begin{cases}
L_a u=0&\text{ in } B^+_{2R}\subset\RR^{n+1}_+\\ 
\displaystyle{\frac{\partial u}{\partial\nu^a}}
=\varphi&\text{ on } \Gamma^0_{2R}.
\end{cases}
\end{equation*}

Then, there exists  $\beta \in (0,1)$ depending only on $n$, $a$, and $\sigma$, 
such that $u \in C^\beta(\overline{B_R^+})$ and 
$y^a u_y \in C^\beta(\overline{B_R^+})$. 

Furthermore, there exist constants $C^1_R$ and $C^2_R$ depending only on $n$, $a$, 
$R$, $\|u\|_{L^\infty(B_{2R}^+)}$ and also on $\|\varphi \|_{L^\infty(\Gamma^0_{2R})}$ (for $C^1_R$) 
and $\|\varphi \|_{C^\sigma(\Gamma^0_{2R})}$ (for $C^2_R)$, such that 
\begin{equation}\label{reg11}
 \|u\|_{C^\beta(\overline{B_R^+})} \leq C^1_R
\end{equation}
and 
\begin{equation}\label{reg12}
 \|y^a u_y\|_{C^\beta(\overline{B_R^+})} \leq C^2_R.
\end{equation}
\end{lem}

Problem \eqref{temp} has variational
structure, with corresponding energy functional
\begin{equation}\label{enerfunc}
E_{B_R^+}(w)=
\int_{B_R^+}\dfrac{1}{2}y^{a} |\nabla w|^2
+\int_{\Gamma^0_R} \frac{1}{1+a} G(w),
\end{equation}
where $G'=-f$.
This allows us to introduce some of the following notions.

\begin{defi} 
\label{defsolns}
{\rm

a) We say that $u$ is a {\it layer solution} of \eqref{extAlpha}
if it is a bounded solution of \eqref{extAlpha},
\begin{equation}
\label{increasing}
u_{x_1}>0\quad \text{on } \partial\R^{n+1}_+, \ \text{ and}
\end{equation}
\begin{equation}
\label{limits}
\lim_{x_1\to\pm\infty}u(x,0)=\pm 1 \quad\text{for every }
(x_2,\ldots ,x_{n})\in \R^{n-1}.
\end{equation}
Note that we will indifferently call layer solution a solution as above for problem 
\eqref{extAlpha} or a solution $v$ of equation \eqref{problem} satisfying the same properties.

b) Assume that $u$ is a $C^\beta$ function  in $\overline{\R^{n+1}_+}$
for some $\beta \in (0,1)$,
satisfying $-1<u<1$ in $\overline{\R^{n+1}_+}$ and such that for all $R>0$,
$$y^a |\nabla u|^2 \in L^1(B_R^+).$$ 
We say that $u$ is a {\it local minimizer} of 
problem~\eqref{extAlpha}~if
$$ 
E_{B_R^+} (u)\le E_{B_R^+} (u+\psi) 
$$ 
for every $R>0$ and every $C^1$ function $\psi$ in $\overline{\R^{n+1}_+}$
with compact support in $B_R^+ \cup \Gamma_R^0$ and 
such that $-1\le u+\psi \le 1$ in $B_R^+$. 
To emphasize this last condition, in some occasions we will 
say that $u$ is a local minimizer relative to
perturbations in $[-1,1]$.

c) We say that $u$ is a {\it stable solution} of~\eqref{extAlpha}
if $u$ is a bounded solution of \eqref{extAlpha} and if
\begin{equation}
\label{stability} 
\int_{\R^{n+1}_+} y^{a} |\nabla\xi|^2 - 
\int_{\partial\R^{n+1}_+} \frac{1}{1+a} f'(u)\, \xi^2 \ge 0 
\end{equation}
for every function $\xi\in C^1(\overline{\R^{n+1}_+})$ 
with compact support in $\overline{\R^{n+1}_+}$. 
} 
\end{defi}

It is clear that every local minimizer is a stable solution.
At the same time, it is not difficult to prove that every layer solution $u$ is also a
stable solution ---for this, one uses Lemma~\ref{polipo} below and the fact that
$u_{x_1}$ is a positive solution of the linearized problem to \eqref{extAlpha}.

\subsection{Layer solutions in $\RR$} 
The following result characterizes the nonlinearities~$f$
for which problem \eqref{problem} admits a layer solution in $\mathbb R$. 
In addition, it contains a result on uniqueness of layer
solutions.  

\begin{theorem}\label{existNonlocal}
Let $f$ be any $C^{1,\gamma}(\RR)$ function with $\gamma >\max(0,1-2s)$, where $s\in (0,1)$.
Let $G'=-f$. 
Then, there exists a solution $v$ of 
$$(-\partial_{xx})^s v =f(v)\,\,\,\,\text{in}\,\,\RR$$
such that $v'>0$ in $\RR$ and $\lim_{x \to \pm \infty} v(x)=\pm 1$
if and only if 
\begin{equation}\label{condG}
G'(-1)=G'(1)=0 \quad \text{and} \quad
G>G(-1)=G(1) \text{ in } (-1,1).
\end{equation}

If in addition $f'(-1)<0$ and $f'(1)<0$, then this solution is unique up to translations.

As a consequence, if $f$ is odd and $f'(\pm 1)<0$, 
then the solution is odd with respect to some
point. That is, $v(x+b)=-v(-x+b)$ for some $b\in\mathbb{R}$.
\end{theorem}

\begin{rema}\label{uniq}
\rm{
The statement on uniqueness of layer solution also holds for any nonlinearity
$f$ of class $C^1([-1,1])$ satisfying $f'(-1)<0$ and $f'(1)<0$.
There is no need for $f'$ to be $C^\gamma([-1,1])$. Indeed, we will see that the proof
follows that of \cite{BHM} and thus only requires $f$ to be Lipschitz in
$[-1,1]$ and nonincreasing in a neighborhood of $-1$ and of $1$.
See also Lemma~5.2 of \cite{CSM} where this more general assumption is presented.
}
\end{rema}

Note that a layer solution $v=v(x)$, $x\in\RR$, as in Theorem~\ref{existNonlocal} 
provides with a family of layer solutions of the same equation in
$\RR^n$. More precisely, for each direction $e\in\RR^n$, with $|e|=1$ and $e_1>0$,
let 
$$
v^e(x_1,\ldots,x_n):=v(\langle e,(x_1,\ldots,x_n)\rangle).
$$ 
Then, $v^e$ is a layer
solution of
\begin{equation}\label{eqstar}
 (-\Delta)^{s} v^e= f(v^e)\,\,\,\text{ in } \mathbb{R}^{n}.
\end{equation}
This fact is not immediate from the definition of the fractional Laplacian
\eqref{deflapl}
through principal values in $\RR$ and in $\RR^n$ ---indeed, the integrals
in $\RR$ and in $\RR^n$ differ, but the normalizing constants $C_{n,s}$
in front make them agree.
This fact ---that $v^e$ solves \eqref{eqstar}---
follows directly from the equivalence of problem \eqref{problem} with
the extension problem \eqref{extAlpha}
and the fact that the constant $d_s$ in \eqref{cttNeumann} is independent of the
dimension~$n$.

The equality $G(-1)=G(1)$ is equivalent to
$$
\int_{-1}^1 f(s)ds = 0.
$$

\begin{rema}\label{consecutive}
\rm{
Note that $G$ may have one or several local 
minima in $(-1,1)$
with higher energy than $-1$ and 1, and still satisfy
condition \eqref{condG}. Such $G$ will therefore admit 
a layer solution, hence a solution with limits $-1$ and 1
at infinity. Instead,
such layer solution will not exist if $G$ has a minimum at some point
in $(-1,1)$ with same height as $-1$ and $1$. 
In particular, when $G$ is periodic (as in the Peierls-Nabarro
problem $f(u)=\sin (\pi u)$, see \cite{T1}), 
the previous theorem proves that there exists no increasing solution connecting two
non-consecutive absolute minima of $G$.
}
\end{rema}

In \cite{PSV}, with different techniques than ours it is proved 
that for potentials $G$ with $G'(-1)=G'(1)=0$ and $G> G(-1)=G(1)$ in $(-1,1)$,   
there exists a layer solution to equation \eqref{problem}. 
We also refer to the interesting paper
\cite{frank} where properties of ground state solutions are investigated.

Our next result gives the asymptotic behavior of layer solutions.

\begin{theorem}\label{asympNL}
Let $f$ be any $C^{1,\gamma}(\RR)$ function with $\gamma >\max(0,1-2s)$, where $s\in (0,1)$.  
Assume that $f'(-1)<0$, $f'(1)<0$, and that $v$ is a layer solution of 
$$(-\partial_{xx})^s v=f(v)\,\,\,\text{in}\,\,\RR.$$
 
Then, there exist constants $0<c\leq C$ such that
\begin{equation}\label{asymptder}
c|x|^{-1-2s}\leq v'(x) \leq C|x|^{-1-2s}\,\,\,\,\text{for}\,\,\,|x|\geq 1. 
\end{equation}
As a consequence, for other constants $0<c\leq C$,
\begin{equation}\label{asympt1}
cx^{-2s}\leq 1-v(x) \leq Cx^{-2s}\,\,\,\,\text{for}\,\,\,x>1
\end{equation}
and
\begin{equation}\label{asympt-1}
c|x|^{-2s}\leq 1+v(x) \leq C|x|^{-2s}\,\,\,\text{for}\,\,\,x<-1.
\end{equation}
\end{theorem}
 
To prove the above theorem for a given nonlinearity $f$,
the following almost explicit layer solution
(we emphasize that it is a layer solution for another nonlinearity) will be very useful.
More properties and remarks on this concrete layers will be given in
section~3.

\begin{theorem}\label{examenintro}
Let $s\in (0,1)$. For every $t>0$, the $C^{\infty}(\RR)$ function 
$$
v^t_s(x):=-1+2\int_{-\infty}^x p_s(t,\overline{x})\,d\overline{x}
=\frac{2}{ \pi}\int_0^\infty\frac{\sin(xr)}{r}e^{-tr^{2s}} \, dr
$$
is the layer solution in $\RR$ of \eqref{problem} for a nonlinearity  $f^t_s\in 
C^{1}([-1,1])$ which is odd and satisfies
$f^t_s(0)=f^t_s(1)=0,$ $f^t_s >0$ in $(0,1)$, and $(f^t_s)'(\pm 1)=-1/t$.
\end{theorem}

In the theorem, since $f^t_s\in C^1([-1,1])$ and 
$(f^t_s)'(\pm 1) <0$, Theorem~\ref{existNonlocal} and Remark~\ref{uniq}
guarantee that 
its corresponding layer $v^t_s$ is unique up to translations.

As we will see in Theorem~\ref{minimality} below, every layer solution is a local minimizer and,
in particular, a stable solution. This holds in any dimension and for any nonlinearity.
Our next result states that the converse is also true in dimension one and under certain hypothesis 
on the nonlinearity. That is,   
under various assumptions on $G$, we prove that, for $n=1$, local minimizers, solutions 
with limits (not monotone a priori), or stable solutions are indeed layer solutions. 

\begin{theorem}
\label{classif} 
Let $f$ be any $C^{1,\gamma}(\RR)$ function, with $\gamma >\max(0,1-2s)$. Let $n=1$ and $u$ be a 
function such that 
$$
|u|<1 \quad\text{ in }\overline{\R^2_+}.
$$

{\rm a)} Assume that $G>G(-1)=G(1)$ in $(-1,1)$,
and that $u$ is a local minimizer of problem~\eqref{extAlpha}
relative to perturbations in $[-1,1]$. 
Then, either $u=u(x,y)$ or $u^*=u^*(x,y):=u(-x,y)$ 
is a layer solution of~\eqref{extAlpha}.

{\rm b)} Assume $G''(-1)>0$, $G''(1)>0$, and that $u$ is a solution 
of~\eqref{extAlpha} with $$\lim_{x\rightarrow\pm\infty}u(x,0)=\pm 1.$$
Then, $u$ is a layer solution of~\eqref{extAlpha}.

{\rm c)} Assume that $G$ satisfies$:$
\begin{eqnarray}\label{Gstable2}
& &\text{if }-1\le L^- < L^+\le 1, \; G'(L^-)=G'(L^+)=0 ,\\
&&
\text{and } \; G>G(L^-)=G(L^+)
\text{ in } (L^-,L^+), \\
\label{Gstable3}
& &\text{then } L^-=-1 \text{ and } L^+=1.
\end{eqnarray}
Let $u$ be a nonconstant stable solution of \eqref{extAlpha}. 
Then, either $u=u(x,y)$ or $u^*=u^*(x,y):=u(-x,y)$
is a layer solution of~\eqref{extAlpha}.
\end{theorem}

\begin{rema}\label{twolayers}
\rm{
Notice that the hypothesis \eqref{Gstable2}-\eqref{Gstable3} on $G$ in part c) of the theorem
is necessary to guarantee that $u$ connects $\pm 1$. Indeed, assume that
$-1<L^-<L^+<1$ are four critical points of $G$ with $G>G(-1)=G(1)$ in $(-1,1)$
and with $G>G(L^-)=G(L^+)$ in $(L^-,L^+)$. Assume also that
\begin{equation*}
G(\pm 1) < G(L^\pm).
\end{equation*}
Then, by our existence result (Theorem~\ref{existNonlocal}) applied twice
---in $(-1,1)$ and also in $(L^-,L^+)$ after rescaling it---, we have
that $(-\partial_{xx})^{s}v=f(v)$ in $\RR$ admits two different increasing
solutions: one connecting $L^\pm$ at infinity, and another connecting $\pm 1$.

Instead, as pointed out in 
Remark~\ref{consecutive}, if $G\geq G(\pm 1) = G(L^\pm)$ in $(-1,1)$, then there is no
increasing solution connecting $\pm 1$, as a consequence of our
Modica estimate, which gives \eqref{nec2}.

Note that an identically constant
function $u\equiv s$ is a stable solution of~\eqref{problem}
if and only if $G'(s)=0$ and $G''(s)\ge 0$.
This follows easily from the definition~\eqref{stability} of
stability. Therefore, regarding part c) of the previous theorem,
a way to guarantee that a stable solution $u$ is nonconstant
is that $u=s\in (-1,1)$ at some point and that either $G'(s)\not =0$
or $G''(s)<0$. 
}
\end{rema}

\subsection{Stability, local minimality, and symmetry of solutions}
The following result states that every layer solution 
in $\mathbb{R}^{n+1}_+$ is
a local minimizer.
This result is true in every dimension $n$. 

\begin{theorem}
\label{minimality}
Let $f$ be any $C^{1,\gamma}(\RR)$ function and $\gamma >\max(0,1-2s)$, where $s\in (0,1)$. 
Assume that problem \eqref{extAlpha} admits a layer solution $u$
in $\mathbb{R}^{n+1}_+$ with $n\geq 1$. Then$\,:$

{\rm a)}
$u$ is a local minimizer of problem \eqref{extAlpha}.

{\rm b)} The potential $G$ satisfies
\begin{equation}\label{balRn}
G'(-1)=G'(1)=0 \quad \text{and} \quad
G\geq G(-1)=G(1) \text{ in } (-1,1).
\end{equation}
\end{theorem}

The strict inequality $G> G(-1)=G(1)$ in \eqref{balRn} is known
to hold when $n=1$ or, as a consequence, when $u(\cdot, 0)$ 
is a one-dimensional solution in $\RR^n$.
We established this in \cite{CS1} (it is one of the implications
in Theorem~\ref{existNonlocal} above). 
The strict inequality $G>G(\pm 1)$ also holds when $n=2$
(as a consequence of Theorem~\ref{symNL} below) and when
$n=3$ and $s\geq1/2$ (as a consequence of a result from \cite{CCin2}).
It remains an open question in the rest of cases.

For $n=2$, we prove that bounded stable solutions $u$
(and hence also local minimizers and layer solutions)  
are functions of only two variables: $y$ and a linear combination 
of $x_1$ and $x_2$. This statement on the 1D symmetry of $u(\cdot,0)$ is closely
related to a conjecture of De~Giorgi on 1D symmetry for interior
reactions, proved in \cite{GG1,AC,AAC} in low dimensions and partially settled by 
Savin~\cite{savin} up to dimension 8. We also refer the reader to \cite{SV1,SV2} 
where some rigidity properties of boundary reactions have been established through a 
more geometric approach.  Particularly, in \cite{SV1}, the 
following symmetry result in dimension 
$n=2$ is proved by using a completely different approach than the one used in the present paper, 
relying on a weighted Poincar\'e inequality (see also \cite{FAR}).

\begin{theorem}\label{symNL}
Let $f$ be any $C^{1,\gamma}(\RR)$ function and $\gamma >\max(0,1-2s)$, where $s\in (0,1)$.
Let $v$ be a bounded solution of  
$$(-\Delta)^s v=f(v)\,\,\,\text{in }\,\,\RR^2 .$$
Assume furthermore that its extension $u$ is stable. 

Then, $v$ is a function of one variable. 
More precisely,
$$
v(x_1,x_2)=v_0\left( \cos(\theta) x_1+\sin(\theta) x_2\right)\quad
\text{in }\mathbb{R}^2
$$
for some angle $\theta$ and some solution $v_0$ 
of the one-dimensional problem with same nonlinearity~$f$, and
with either $v'_0 > 0$ everywhere or
$v_0$ identically constant.
\end{theorem}

For $n=3$ and $s\geq 1/2$, this 1D symmetry result has been proved by 
E. Cinti and one of the authors in \cite{CCin1,CCin2}.
It remains open for $n=3$ and $s< 1/2$, and also for $n\geq 4$.

A simpler task than the study of all stable solutions consists of 
studying solutions 
$u$ of \eqref{extAlpha} with $|u|\le 1$ and satisfying the limits
\eqref{limits} {\it uniformly} in $(x_2,\ldots ,x_{n})\in\mathbb{R}^{n-1}$.
Under hypothesis $f'(-1)<0$ and $f'(-1)<0$, 
it is possible to establish in every dimension~$n$ that 
these solutions depend only on
the $y$ and $x_1$ variables, and are monotone 
in~$x_1$. 
Here, by the uniform limits hypothesis,
the $x$-variable in which the solution finally depends on is known
a priori ---in contrast with the variable of dependence in Theorem~\ref{symNL}.
For the standard Laplacian this result was first established in~\cite{BHM} 
using the sliding method.
We will not provide the proof of the result because it is completely analogue to the one  
in \cite{CSM}. Since our operator $L_a$ is 
invariant under translations in $x$, one can perform the sliding method together with the 
maximum principles proved in \cite{CS1}.

Theorem~\ref{symNL} is a partial converse in dimension two
of Theorem~\ref{minimality} a), 
in the sense that it establishes the monotonicity of
stable solutions and in particular, of local minimizers.
The remaining property for being a layer solution
(i.e., having limits $\pm 1$ at infinity) 
requires additional hypotheses on $G$, as in Theorem~\ref{classif}.

\subsection{Outline of the paper} In section 3 we construct an almost explicit 
layer solution (Theorem~\ref{examenintro})
and we use it to establish the asymptotic behavior of any layer solution in $\RR$
as stated in Theorem~\ref{asympNL}. 
In section 4 we prove the existence of minimizers to mixed Dirichlet-Neumann problems
in bounded domains of $\RR^{n+1}_+$ ---a result needed in subsequent sections.  
In section 5 we prove the local minimality of layer solutions in any dimension 
and the necessary conditions
on $G$ for such a layer in $\R^n$ to exist, Theorem \ref{minimality}. 
The 1D symmetry result for stable solutions in $\RR^2$,
Theorem~\ref{symNL}, is established in section 6. 
Finally, section~7 concerns layers in $\RR$ and establishes the existence 
Theorem~\ref{existNonlocal} 
and the classification result Theorem~\ref{classif}.

\section{An example of layer solution. Asymptotic properties of layer solutions}

In this section we provide with an example of layer solution based on the fractional heat equation. 
From it, we get the asymptotic behavior of layers for all other nonlinearities.
Let us first explain how the concrete layer is found. 

The starting point is the fractional heat equation, 
\begin{equation}\label{fracHeat}
\partial_t w+(-\partial_{xx})^s w=0,  \qquad t>0,\,\,\,x \in \RR ,
\end{equation}
which is known to have a fundamental solution of the form
\begin{equation}\label{explicitq}
p_s(t,x)=t^{-\frac{1}{2s}}q_{s}(t^{-\frac{1}{2s}} x) >0
\end{equation}
for $x\in\RR, t>0$. 
Being the fundamental solution, $p_s$ has total integral
in $x$ equal to $1$, i.e., 
\begin{equation}\label{propps}
\int_{\RR}p_{s}(t,x)\,dx=1 \qquad \text{ for all } t>0.
\end{equation}
To compute $p_s$, one takes the Fourier transform of \eqref{fracHeat} to obtain
$$
\partial_t \widehat{p_s} +|\xi|^{2s}\widehat{p_s} =0,
$$
where $\widehat{p_s}=\widehat{p_s}(t,\xi)$ is the Fourier transform in $x$ of $p_s(t,x)$.
Thus, since $p_s(0,\cdot)$ is the Delta at zero and hence $\widehat{p_s}(0,\cdot)\equiv 1$,
we deduce
$$
\widehat{p_s}(t,\xi) =\exp \{-t|\xi|^{2s}\}.
$$
{From} this, by the inversion formula for the Fourier transform, we find
\begin{equation}\label{explicitp}
 p_s(t,x)=\frac{1}{ \pi}\int_0^\infty \cos(xr)e^{-tr^{2s}}\,dr.
\end{equation}

It follows that the function 
\begin{equation}\label{defv}
 v^t_s(x):=-1+2\int_{-\infty}^x p_s(t,\overline{x})\,d\overline{x}
=2\int_0^x p_s(t,\overline{x})\,d\overline{x}
\end{equation}
is increasing and has limits $\pm 1$ at $\pm \infty$.
The concrete expression \eqref{explicitv} below for $v^t_s$ 
is obtained by interchanging the order of the two integrals
when using \eqref{explicitp} to compute the primitive of $p_s$.

That $v^t_s$ is a layer solution is stated in the next
theorem, which contains all statements in~Theorem~\ref{examenintro}
and also the asymptotic behavior of $v^t_s$, among other facts. 
The proof of the theorem is given at the end of this section.

\begin{theorem}\label{examen}
Let $s\in (0,1)$. For every $t>0$, the $C^\infty(\RR)$ function 
\begin{eqnarray}\label{explicitv}
v^t_s(x) &:= & -1+2\int_{-\infty}^x p_s(t,\overline{x})\,d\overline{x}
= 2\int_0^x p_s(t,\overline{x})\,d\overline{x}\nonumber\\
&=& \frac{2}{ \pi}\int_0^\infty\frac{\sin(xr)}{r}e^{-tr^{2s}} \, dr\nonumber\\
&=& \textrm{sign}(x)\, \frac{2}{ \pi}\int_0^\infty\frac{\sin(z)}{z}e^{-t(z/|x|)^{2s}} \, dz
\end{eqnarray}
is the layer solution in $\RR$ of \eqref{problem} for a nonlinearity  $f^t_s\in 
C^{1}([-1,1])$ which is odd and twice differentiable in $[-1,1]$  and which satisfies 
$$
f^t_s(0)=f^t_s(1)=0 , \quad f^t_s >0 \text{ in } (0,1), \quad (f^t_s)'(\pm 1)=-\,\frac{1}{t} 
$$
and
\begin{equation}\label{secondf}
(f^t_s)''(1)=-\, \frac{\pi}{t}\, \frac{\cos(\pi s)}{\sin(\pi s)}\,\frac{\Gamma(4s)}{(\Gamma(2s))^2}\,\,
\begin{cases}
<0 &\text{ if } 0<s< 1/2 ,\\ 
=0 &\text{ if }  s =1/2, \\
>0 &\text{ if } 1/2< s <1. 
\end{cases}
\end{equation}

In addition, the following limits exist:
\begin{equation}\label{limder}
\lim_{|x| \to \infty} |x|^{1+2s}(\partial_x v^t_s)(x) =t \frac{4s}{\pi}\sin(\pi s)\Gamma(2s)>0
\end{equation}
and, as a consequence, 
$$\lim_{x \to \pm \infty} |x|^{2s} |v^t_s(x) \mp 1| =t \frac{2}{\pi}\sin(\pi s)\Gamma(2s)>0. $$
\end{theorem}

\begin{rema}\label{remark32}
\rm{
As stated in the theorem, we have $f^t_s\in C^1([-1,1])$ and 
$(f^t_s)'(\pm 1) <0$ for every $s\in (0,1)$. 
In particular, by Theorem~\ref{existNonlocal} and Remark~\ref{uniq}, 
its corresponding layer $v^t_s$ is unique up to translations.

When $s=1/2$, the particular layer above agrees with the explicit one used in 
\cite{CSM}, namely
$$
v^t_{1/2}(x)=\frac{2}{\pi}\arctan\frac{x}{t}, \quad\text{with }
\,\, f^t_{1/2}(v)=\frac{1}{\pi t}\sin (\pi v).
$$
This can be easily seen computing \eqref{explicitp} explicitly when $s=1/2$,
using integration by parts, to obtain 
$$
\partial_x v^t_{1/2}(x)= 2 p_{1/2}(t,x)=
\frac{2}{\pi}\,\frac{1}{t}\,\frac{1}{1+x^2/t^2}.
$$

We may try to see which function we obtain in the above formulas setting
$s=1$. In this case, \eqref{explicitp} can be checked to be equal to a Gaussian and thus
$v^t_1$, two times its primitive, is the error function $\text{erf}(x)$ ---up to 
a scaling constant. Its derivative is therefore
$e^{-cx^2}$, which does not have the correct decay $e^{-cx}$ at $+\infty$ for 
the derivative $v'$ of a layer
solution to $-v''=f(v)$. This is due to the fact that the limit as $s\to 1$ of $f^t_s$
will not be a $C^1([-1,1])$ nonlinearity at the value $1$ ---even if they
all satisfy $(f^t_s)'(1)=-1/t$. The reason is that their second derivatives at $1$,
$(f^t_s)''(1)$, blow-up as $s\to 1$ as shown by \eqref{secondf}.

Note also that \eqref{secondf} shows that, when $1/2<s<1$, the nonlinearity $f^t_s$ 
is positive but not concave in $(0,1)$.
}
\end{rema}

The following immediate consequence of Theorem \ref{examen} will give the
asymptotic behavior of layer solutions for any nonlinearity $f$. 

\begin{coro}\label{supnl}
Let $s \in (0,1)$ and $t>0$ be a constant. Then, the function 
$$
\varphi^t =\partial_x v^t_s >0, 
$$
where $v^t_s$ is the explicit layer of Theorem~\ref{examen},
satisfies
\begin{equation}\label{vtsuper}
(-\partial_{xx})^s \varphi^t (x)+ 2t^{-1} \varphi^t (x) \geq 0 \qquad\text{for }
\ |x|\text{ large enough,}
\end{equation}
\begin{equation}\label{vtsub}
(-\partial_{xx})^s \varphi^t (x)+ 2^{-1}t^{-1} \varphi^t (x) \leq 0 \qquad\text{for }
\ |x|\text{ large enough,}
\end{equation}
and that the following limit exists and is positive: 
\begin{equation}\label{vtlim}
\lim_{|x| \to \infty}  |x|^{1+2s} \varphi^t (x)\in (0, +\infty).
\end{equation}
\end{coro}

\begin{proof}
Clearly $\varphi^t =\partial_x v^t_s>0$ satisfies the linearized equation
$$
(-\partial_{xx})^s \varphi^t - (f^t_s)'(v^t_s(x)) \, \varphi^t=0\quad\text{ in } \RR.
$$
Using that  $\varphi^t >0$, $v^t_s$ has limits $\pm 1$ at $\pm\infty$,
$f^t_s$ is $C^1([-1,1])$, 
and that $(f^t_s)'(\pm 1) =-1/t$, both \eqref{vtsuper} and \eqref{vtsub} follow.
The statement \eqref{vtlim} follows from \eqref{limder}.
\end{proof}

With this corollary at hand, we can now prove the asymptotics of any layer.

\medskip
\noindent
{\it Proof of Theorem}~\ref{asympNL}.
The proof uses Corollary~\ref{supnl} above and a very easy maximum principle,
Lemma 4.13 and Remark 4.14 of \cite{CS1}. Its statement in dimension one
is the following. 

Let $w\in C^2_{\rm loc}(\RR)$ be a continuous function in $\R$ such that
$w(x)\rightarrow 0$ as $|x|\rightarrow\infty$ and
\begin{equation}\label{linMP}
(-\partial_{xx})^s w +d(x) w\geq 0 \qquad \text{in }\R
\end{equation}
for some bounded function $d$. Assume also that, for some
nonempty closed set $H\subset\R$, one has
$w> 0$ in $H$ and that $d$ is continuous and nonnegative in $\R\setminus H$.
Then, $w>0$ in $\R$.

Let now $f$ and $v$ be a nonlinearity and a layer as in Theorem~\ref{asympNL}.
We then have
\begin{equation}\label{linv}
(-\partial_{xx})^s v' - f'(v) v'=0 \quad\text{ in } \R.
\end{equation}

To prove the upper bound for $v'$ in \eqref{asymptder}, we take $t$ large enough
such that
$2t^{-1}<\min\{-f'(-1),-f'(1)\}$. Then, for any positive constant $C>0$,
$$
w:=C\varphi^t - v'
$$
satisfies, by \eqref{vtsuper} and \eqref{linv}, $(-\partial_{xx})^s w+ 2t^{-1} w \geq 0$ 
for $|x|$ large enough, say for $x$ in the complement of a compact interval $H$.
Next, take the constant $C>0$ so that $w\geq 1$ in the compact set $H$, and
define now $d$ in $H$ so that $(-\partial_{xx})^s w+ d w = 0$ in $H$
---recall that $w\geq 1$ in $H$ and hence $d$ is well defined and bounded in $H$.  
We take $d= 2t^{-1}$ in $\R\setminus H$.
Thus, \eqref{linMP} is satisfied and, since $w\to 0$ at infinity, 
the maximum principle above leads to
$w>0$ in $\R$. This is the desired upper bound for $v'$ in \eqref{asymptder},
since $\varphi^t$ satisfies \eqref{vtlim}.

To prove the lower bound for $v'$ in \eqref{asymptder}, we proceed in the same way but
replacing the roles of $v'$ and $\varphi^t$. For this, we now take $t>0$ small enough
such that
$\max\{-f'(-1),-f'(1)\}<2^{-1}t^{-1}$. Thus,
$\tilde w:=C v'- \varphi^t$ satisfies $(-\partial_{xx})^s \tilde w+ 2^{-1}t^{-1} \tilde w 
\geq 0$ for $|x|$ large enough. One proceeds exactly as before to obtain
$\tilde w>0$ in $\R$ for $C$ large enough, 
which is the desired lower bound for $v'$ in \eqref{asymptder}.
\hfill\qed
\medskip

It remains to establish Theorem \ref{examen}. 
For this, we use the following well-known technical lemma due to G. P\'olya \cite{Pol},
1923. We prove it here for completeness; in fact, the proof as explained in \cite{Pol}
only works for $s\leq 1/2$. For $s>1/2$, we follow the proof given in \cite{Kol}.

\begin{lem}\label{calcul}
For $\kappa > 0 $ and $s \in (0,1)$, we have 
$$\lim_{x \to +\infty} \int_0^\infty \sin(z) z^{\kappa s-1}e^{-(z/x)^{2s}}\,dz
=\sin(\kappa s \pi/2)\Gamma(\kappa s). $$
\end{lem}

\begin{proof}
For every $x>0$, we have 
\begin{equation*}
\begin{split}
\int_0^\infty \sin(z) z^{\kappa s-1}e^{-(z/x)^{2s}}\,dz &=\text{Im}\int_0^\infty 
 e^{iz}z^{\kappa s-1}e^{-(z/x)^{2s}}\,dz \\
&=\text{Im} \int_0^\infty h_x(z) \,dz ,
\end{split}
\end{equation*}
where 
$$h_x(z):= z^{\kappa s-1}e^{iz-(z/x)^{2s}}. $$
Let us also denote
$$h_\infty(z):= z^{\kappa s-1}e^{iz}. $$

For $0\leq \theta\leq \pi/2$, let $\gamma_\theta$ be the half-line from the origin
making an angle $\theta$ with the positive $x$-axis. We will next see that,
for certain angles $\theta$,
$\text{Im}\int_{\gamma_\theta} h_x(z)\,dz$ are all equal and independent of those $\theta$. 
For this, given two angles $0\leq \theta_1 < \theta_ 2\leq \pi/2$
and $R>0$, we integrate counter-clockwise on the contour given by
the segments of length $R$ starting from $0$ on  $\gamma_{\theta_1}$ and on
$\gamma_{\theta_2}$, and by the arc $\Gamma^R_{\theta_1,\theta_2}$
of radius $R$ with center at the origin and
joining the two end points of the previous segments. We also need to remove a 
neighborhood of zero and add a
small arc  with center at the origin connecting the two half-lines.
The integrals of $h_x$ and of $h_\infty$ in this small arc will tend to zero as the radius
tends to zero, since $|h_x(z)|+|h_\infty(z)|\leq C|z|^{\kappa s-1}$ near the origin.

The key point is to make sure that the integral of $h_x$, and later of $h_\infty$,
on the arc $\Gamma^R_{\theta_1,\theta_2}$ of radius $R$ tends to zero as $R\to\infty$ 
if we choose the angles $0\leq \theta_1 < \theta_ 2\leq \pi/2$ correctly. 
Note that if $z\in \CC$ belongs to such an arc, then $z$ belongs to the sector
$$
S_{\theta_1,\theta_2}:=\{z\in \CC\, :\, \theta_1\leq \text{Arg}(z)\leq \theta_2\}.
$$
To guarantee the convergence to zero of the integral on the arc, note that
\begin{equation}\label{boundarcs}
|h_x(z)|= |z|^{\kappa s-1}\exp\{-\text{Im}(z)-x^{-2s}\text{Re}(z^{2s})\}
\end{equation}
and
\begin{equation}\label{boundinfarc}
|h_\infty(z)|= |z|^{\kappa s-1}\exp\{-\text{Im}(z)\}
\end{equation}
for all $z\in\CC$ in the first quadrant.

We need to distinguish two cases.

\smallskip
{\it Case}~1. Suppose that $s\leq 1/2$. In this case we take 
$\theta_1=0$ and $\theta_2=\pi/2$. Then, if $z$ lies in the sector $S_{0,\pi/2}$
(the first quadrant), then $z^{2s}$ is also in the first quadrant, since $2s\leq 1$. 
Thus, the real and imaginary parts appearing in \eqref{boundarcs}  are both nonnegative,
and at least one of them positive up to the boundary of the quadrant. Thus, by  \eqref{boundarcs},
$|h_x|\to 0$ exponentially fast ---as $\exp\{-c(x)|z|^{2s}\}$--- uniformly in all the quadrant.
Hence, the integral on the arc $\Gamma^R_{0,\pi/2}$ tends to zero as $R\to\infty$.
We deduce
\begin{equation*}
\begin{split}
\int_0^\infty \sin(z) z^{\kappa s-1}e^{-(z/x)^{2s}}\,dz &=
\text{Im} \int_{\gamma_0} h_x(z) \,dz = \text{Im} \int_{\gamma_{\pi/2}} h_x(z) \,dz \\
 &= \text{Im}\left\{  e^{i \kappa s \pi /2}  \int_0^\infty  y^{\kappa s-1}
e^{-y -i^{2s}(y/x)^{2s}}\,dy \right\}.
\end{split}
\end{equation*}
Note that the function in the last integral
is integrable since 
$$
|e^{-y -i^{2s}(y/x)^{2s}}|=|e^{-y -(\cos(s\pi)+i\sin(s\pi))(y/x)^{2s}}|
=e^{-y -\cos(s\pi)(y/x)^{2s}}\leq e^{-y}
$$ 
due to $s\leq 1/2$.
Thus, the limit as $x\to +\infty$ exists and is equal to
\begin{eqnarray}\label{case1}
\lim_{x\to +\infty} \int_0^\infty \sin(z) z^{\kappa s-1}e^{-(z/x)^{2s}}\,dz 
&=& \text{Im}\left\{  e^{i \kappa s \pi /2}  \int_0^\infty  y^{\kappa s-1}
e^{-y} \,dy \right\} \nonumber \\
&=& \sin( \kappa s \pi /2)  \int_0^\infty  y^{\kappa s-1}
e^{-y} \,dy \nonumber \\
&=& \sin(\kappa s \pi/2)\Gamma(\kappa s),
\end{eqnarray}
as claimed.

\smallskip
{\it Case}~2. Suppose now that $1/2<s<1$. In this case \eqref{boundarcs} does not tend
to zero at infinity in all the first quadrant, since $2s>1$ and thus $\text{Re}(z^{2s})$
becomes negative somewhere in the quadrant. Here, we need to take
$$
\theta_1 =0 \quad\text{ and }\quad \theta_2=\frac{\pi}{4s}.
$$ 
Now, in the sector $S_{0,\pi/(4s)}$, the real and imaginary parts appearing 
in \eqref{boundarcs}  are both nonnegative,
and at least one of them positive up to the boundary of the sector. Thus, as before, 
we now deduce
\begin{equation*}
\begin{split}
\int_0^\infty \sin(z) z^{\kappa s-1}e^{-(z/x)^{2s}}\,dz &=
\text{Im} \int_{\gamma_0} h_x(z) \,dz = \text{Im} \int_{\gamma_{\pi/(4s)}} h_x(z) \,dz .
\end{split}
\end{equation*}
Note that in the last integral on $\gamma_{\pi/(4s)}$, we have
\begin{eqnarray*}
|h_x(z)| &= & |z|^{\kappa s-1}\exp\{-\text{Im}(z)-x^{-2s}\text{Re}(z^{2s})\} \\
&=& |z|^{\kappa s-1}\exp\{-\text{Im}(z)\}=|h_\infty(z)|
\end{eqnarray*}
for $z\in \gamma_{\pi/(4s)}$. Besides, by the last expression, $h_\infty$
is integrable on $\gamma_{\pi/(4s)}$. Thus, by dominated convergence, we have
\begin{equation}\label{case2}
\lim_{x\to +\infty} \int_0^\infty \sin(z) z^{\kappa s-1}e^{-(z/x)^{2s}}\,dz =
\text{Im} \int_{\gamma_{\pi/(4s)}} h_\infty(z) \,dz .
\end{equation}

Finally, for this last integral we work on the sector $S_{\pi/(4s), \pi/2}$.
By \eqref{boundinfarc}, $h_\infty(z)$ tends to zero exponentially fast and uniformly
as $|z|\to\infty$ on the sector. Thus,
\begin{eqnarray*}
\text{Im} \int_{\gamma_{\pi/(4s)}} h_\infty(z) \,dz &=&
\text{Im} \int_{\gamma_{\pi/2}} h_\infty(z) \,dz \\
&=& \text{Im}\left\{  e^{i \kappa s \pi /2}  \int_0^\infty  y^{\kappa s-1}
e^{-y} \,dy \right\}.
\end{eqnarray*}
Recalling \eqref{case2}, one concludes as in \eqref{case1}.
\end{proof}

Finally, we can prove our results on the explicit layer.

\medskip
\noindent
{\it Proof of Theorem}~\ref{examen}.
Let $v^t_s$ be defined by \eqref{explicitv}. It is clear that 
\begin{equation}\label{formt}
 v^t_s(x)=v^1_s(t^{-1/(2s)}x).
\end{equation}
Hence, by the definition \eqref{deflapl}
of the fractional Laplacian,
we have $(-\partial_{xx})^sv^t_s(x)=t^{-1}(-\partial_{xx})^sv^1_s(t^{-1/(2s)}x)$.
Thus, having proved all the statements for $v^1_s$, they will also hold for
$v^t_s$ with nonlinearity $f^t_s(v)=t^{-1}f^1_s(v)$.

Hence, we may take $t=1$. To simplify notation, we denote 
$$
v:=v^1_s \qquad\text{ and }\qquad f:=f^1_s.
$$
{From} $v'(x)=2p_s(1,x)$ and expression \eqref{explicitp}, it is clear that
$v\in C^\infty(\RR)$. By expression \eqref{defv}, we have $v(-\infty)=-1$.  
Since $v'=2q_s=2p_s(1,\cdot)>0$, $v$ is increasing.

The fact that $v(+\infty)=1$ is a consequence of \eqref{propps}, 
$\int_\R p_s(1,y)\,dy=1$.
It also follows from expression \eqref{explicitv}
and the well-known fact that, in principal value sense, 
$\int_0^\infty\sin(z)z^{-1}\, dz =\pi/2$.
This can also be proved adding a factor $z^{\kappa s}$ in the integral
\eqref{explicitv}, using then Lemma~\ref{calcul} and noting that 
$\sin(\kappa s \pi/2)\Gamma(\kappa s)=\sin(\kappa s \pi/2)(\kappa s)^{-1}
\Gamma(\kappa s+1)\to \pi/2$ as $\kappa\downarrow 0$.
 
We now prove that there exists a function $f$ such that 
$$(-\partial_{xx})^s v=f(v) \,\,\,\,\text{in}\,\,\RR.$$
For this, we use the expression \eqref{explicitq} and that $p_s$ solves the fractional
heat equation \eqref{fracHeat}.  Because of the commutation of the derivative with the 
fractional Laplacian, we deduce
\begin{eqnarray*}
\{ (-\partial_{xx})^s v\}'(x) &=& (-\partial_{xx})^s v' (x)=2(-\partial_{xx})^s q_s (x) =
-2\partial_t p_s(1,x)\\
&=& \frac{1}{s}\left\{ q_s(x) + x q_s'(x)\right\}.
\end{eqnarray*}
Therefore, integrating by parts,  
\begin{eqnarray*}
(-\partial_{xx})^s v(x)
= \frac{1}{s}\int_{-\infty}^x \{q_s(z) + z q_s'(z) \}\,dz
=\frac{1}{s} x q_s(x)=\frac{1}{2s} x v'(x).
\end{eqnarray*}

Since $v'>0$, the $C^\infty$ function $v=v(x)$ is invertible on $\RR$, with inverse $x=x(v)$,
a $C^\infty$ function on the open interval $(-1,1)$. 
We now set 
\begin{equation}\label{deff}
 f(v):=\frac{1}{2s} x(v) v'(x(v)),
\end{equation}
so that our semilinear fractional equation is satisfied.
We know that $f\in C^\infty (-1,1)$. Also, since $v$ is an odd function, 
its inverse $x$ is also odd and therefore $f$ is odd, by \eqref{deff}.
This expression also gives that $f>0$ in $(0,1)$. 

It remains to verify that $f \in C^1([-1,1])$ once we set $f(\pm 1)=0$ and $f'(\pm 1)=-1$,
and that $f$ is twice differentiable in $[-1,1]$ and having  values for $f''(\pm 1)$ 
given by \eqref{secondf} with $t=1$. 
It also remains to establish the asymptotic behavior of $v'$. 

For all this, using \eqref{explicitp} we compute 
\begin{eqnarray}\label{lima}
(\pi/2) v'(x) &=& \pi q_s(x)=\frac{1}{x} \int_0^\infty \cos(z) e^{-(z/x)^{2s}}\,dz\nonumber \\
&=& \frac{1}{x} \int_0^\infty \{\sin(z)\}' e^{-(z/x)^{2s}}\,dz \nonumber\\
&=& 2s x^{-1-2s} \int_0^\infty \sin (z) z^{2s-1}e^{-(z/x)^{2s}}\,dz ,
\end{eqnarray}
by integration by parts. Hence using Lemma \ref{calcul} with $\kappa=2$, we deduce 
\begin{equation}\label{limab}
\lim_{x \to +\infty} x^{1+2s}v'(x)=\lim_{x \to +\infty} 2x^{1+2s}q_s(x)=
\frac{4s}{\pi}\sin(\pi s)\Gamma(2s),
\end{equation}
as claimed in \eqref{limder} ---for other values of $t$, simply use \eqref{formt}. 
In particular, $\lim_{x \to +\infty} xv'(x)=0$ and thus, by \eqref{deff},
$f$ is continuous on $[-1,1]$ defining $f(\pm 1)=0$.
In addition, we also deduce
\begin{equation}\label{limapr}
 1-v(x)=\int_x^\infty v'(y)\,dy= \frac{2}{\pi}\sin(\pi s)\Gamma(2s)   x^{-2s}+\text{o}(x^{-2s})
\end{equation}
as $x\to +\infty$.

Next, we differentiate \eqref{deff},
that is, $f(v(x))=(2s)^{-1}xv'(x)=(2s)^{-1}x2q_s(x)$, to obtain
$$
f'(v)v'=\frac{1}{2s}\left \{ v'+x(v) 2q_s'(x(v)) \right \}
$$
and hence
\begin{equation}\label{expfprime}
 f'(v)=\frac{1}{2s}\left \{ 1+x(v)\frac{q'_s(x(v))}{q_s(x(v))} \right \}. 
\end{equation}

Thus, using \eqref{explicitp} we compute
\begin{eqnarray*}
\pi x q'_s(x) &=& -\int_0^\infty xr \sin(xr) e^{-r^{2s}}\,dr\\
&=& -x^{-1}\int_0^\infty z\sin(z)e^{-(z/x)^{2s}}\,dz\\
&= & -x^{-1}\int_0^\infty \{ \sin(z)-z\cos(z)\}'e^{-(z/x)^{2s}}\,dz\\
&= & -2s x^{-1-2s}\int_0^\infty \{\sin(z)-z\cos(z)\}z^{2s-1} e^{-(z/x)^{2s}}\,dz. 
\end{eqnarray*}
We also compute $\pi \{ (1+2s)q_s+xq'_s\}$ by adding \eqref{lima} (multiplied by
$1+2s$) to the previous expression. Integrating by parts,  
and at the end invoking Lemma~\ref{calcul}
with $\kappa=4s$, we obtain
\begin{eqnarray}\label{limc}
& & \hspace{-1.5cm} \pi \{ (1+2s)q_s+xq'_s\} = \nonumber \\
&=&   2sx^{-1-2s} \int_0^\infty 
\{2s\sin(z)+z\cos(z)\} z^{2s-1} e^{-(z/x)^{2s}}\,dz\nonumber\\
 &=& 2sx^{-1-2s} \int_0^\infty \{\sin(z)z^{2s}\}'e^{-(z/x)^{2s}}\,dz\nonumber\\
 &=& (2s)^2 x^{-1-4s} \int_0^\infty \sin(z)z^{4s-1} e^{-(z/x)^{2s}}\,dz\nonumber\\
&=& x^{-1-4s}\{4s^2\sin(2\pi s) \Gamma (4s)+\text{o}(1)\}\nonumber\\
&=& x^{-1-4s}\{8s^2\sin(\pi s) \cos(\pi s)\Gamma (4s)+\text{o}(1)\}
\end{eqnarray}
as $x \to +\infty$.

Therefore, from \eqref{expfprime}, \eqref{limc}, and \eqref{limab}, one has
\begin{equation}\label{formfpri}
 f'(v(x)) = -1+\frac{1}{2s}\,\, \frac{(1+2s)q_s+xq_s'}{q_s}=-1+\text{O}(x^{-2s}).
\end{equation} 
Thus, setting $f'(\pm 1)=-1$ and using that $f'$ is even, 
we have that $f$ is differentiable at $\pm 1$. 

Finally, using \eqref{formfpri}, \eqref{limc}, \eqref{limapr}, and \eqref{limab}, we have 
\begin{eqnarray*}
& &  \hspace{-2cm}\frac{f'(v(x))-f'(1)}{v(x)-1} = \frac{f'(v(x))+1}{v(x)-1}\nonumber \\ 
 &=&\frac{1}{2s}\,\, \frac{(1+2s)q_s+xq_s'}{(v-1)q_s}
\to -\pi\frac{\cos(\pi s) \Gamma(4s)}{\sin(\pi s)(\Gamma(2s))^2}
\end{eqnarray*}
as $x \to +\infty$. This establishes that  $f\in C^1([-1,1])$ and also that
$f$ is twice differentiable in all of $[-1,1]$ with
\begin{equation}\label{values2}
f''(\pm 1)=\mp \pi \frac{\cos(\pi s)}{\sin(\pi s)}\,\frac{\Gamma(4s)}{(\Gamma(2s))^2}.
\end{equation}
The proof is now complete.
\hfill\qed

\section{Minimizers of the Dirichlet-Neumann problem in bounded domains}

In this section, we concentrate on the existence of absolute
minimizers of the functional
$E_\Omega(u)$ on bounded domains $\Omega$. 
This is an important step since, as in \cite{CSM}, the existence theory of
layer solutions goes through a localization argument in half-balls of
$\R^{n+1}_+$. 

Let $\Omega\subset\R^{n+1}_+$ be a bounded Lipschitz domain.
We define the following subsets of $\partial\Omega$:
\begin{align}
\label{notation1}
& \partial^0\Omega= \{(x,0)\in \partial\R^{n+1}_+ \, : \,
B_\epsilon^+(x,0)\subset\Omega \, \text{ for some }
\epsilon>0\} \quad
\text{ and}
\\  
\label{notation2}
& \partial^+\Omega=\overline{\partial\Omega\cap\R^{n+1}_+}.
\end{align}

Let $H^1(\Omega,y^a)$ denote the weighted Sobolev space 
$$H^1(\Omega,y^a) =\left \{ u: \Omega \rightarrow \RR\, : \, y^a (u^2 + 
|\nabla u|^2) \in L^1(\Omega)\right \}$$ 
endowed with its usual norm.

Let $u\in C^\beta({\overline\Omega})\cap H^1(\Omega,y^a)$ be a given function
with $|u|\le 1$, where $\beta\in (0,1)$. 
We consider the energy functional
\begin{equation}\label{enerfunct}
E_{\Omega}(v)=
\int_{\Omega}\dfrac{y^a}{2}|\nabla v|^2
+\int_{\partial^0\Omega}\frac{1}{1+a}G(v) 
\end{equation}
in the class 
$$
{\mathcal C}_{u,a}(\Omega)=\{v \in H^1(\Omega,y^a) : -1 \le v \le 1 
\text{ a.e. in }\Omega  
\text{ and } v\equiv u \text{ on }\partial^+\Omega\},
$$ 
which contains $u$ and thus is nonempty. 

The set ${\mathcal C}_{u,a}(\Omega)$ is a closed convex subset of the affine space 
\begin{equation}
\label{defaffine}
H_{u,a}(\Omega)=\{v \in H^1(\Omega,y^a) : v\equiv u \text{ on }\partial^+\Omega\},
\end{equation}
where the last condition should be understood as that
$v-u$ vanishes on $\partial^+\Omega$ in the weak sense.

\begin{lem}\label{existmin}
Let $n\geq 1$ and $\Omega\subset\R^{n+1}_+$ be a bounded Lipschitz domain. 
Let $u\in C^\beta({\overline\Omega})\cap H^1(\Omega,y^a)$
be a given function with $|u|\le 1$, where $\beta\in (0,1)$. 
Assume that 
\begin{equation}\label{subsuper}
f(1)\le 0 \le f(-1).
\end{equation}

Then, the functional
$E_{\Omega}$ admits an absolute minimizer $w$ in ${\mathcal C}_{u,a}(\Omega)$.
In particular, $w$ is a weak solution of 
\begin{equation}
\label{probw}
\begin{cases}
L_a w=0&\text{ in } \Omega\\ 
(1+a) \dfrac{\partial w}{\partial\nu^a} =f(w)&\text{ on }\partial^0\Omega \\
w=u &\text{ on }\partial^+\Omega .
\end{cases}
\end{equation}
Moreover, $w$ is a stable solution of \eqref{probw},
in the sense that
\begin{equation}
\label{stabw}
\int_\Omega y^a |\nabla\xi|^2 - \int_{\partial^0\Omega} \frac{1}{1+a}f'(w) \xi^2 \ge 0
\end{equation}
for every $\xi \in H^1(\Omega, y^a)$ 
such that $\xi\equiv 0$ on 
$\partial^+\Omega$ in the weak sense.
\end{lem}

Hypothesis \eqref{subsuper} states simply that
$-1$ and 1 are a subsolution and a supersolution, respectively,
of~\eqref{probw}.

\medskip
\noindent
{\it Proof of Lemma}~\ref{existmin}.
As in \cite{CSM}, it is useful to consider the following continuous extension $\tilde f$ of
$f$ outside $[-1,1]$:
\begin{equation*}
\tilde f (t)=
\begin{cases}
f(-1)&\text{ if } s\le -1\\ 
f(s) &\text{ if } -1\le s \le 1\\
f(1) &\text{ if } 1\le s . 
\end{cases}
\end{equation*}
Let
$$
\tilde G (s)=-\int_0^s \tilde f,
$$
and consider the new functional
$$
\tilde E_{\Omega}(v)=
\int_{\Omega}\dfrac{y^a }{2}|\nabla v|^2
+\int_{\partial^0\Omega}\frac{1}{1+a}\tilde G(v), 
$$
in the affine space
$H_{u,a}(\Omega)$ defined by \eqref{defaffine}.

Note that $\tilde G=G$ in $[-1,1]$, up to an additive constant.
Therefore, any minimizer $w$ of $\tilde E_\Omega$ in $H_{u,a}(\Omega)$ such
that $-1\le w \le 1$ is also
a minimizer of $E_\Omega$ in ${\mathcal C}_{u,a}(\Omega)$.

To show that  $\tilde E_\Omega$ admits a minimizer in
$H_{u,a}(\Omega)$, we use a standard compactness argument. Indeed, let $v\in H_{u,a}(\Omega)$. 
Since $v-u\equiv 0$ on $\partial^+\Omega$,
we can extend $v-u$ to be identically $0$ in $\R^{n+1}_+\setminus 
\Omega$, and we have $v-u\in H^1 (\R^{n+1}_+,y^a)$. By Nekvinda's result \cite{nekvinda}, 
the trace space of $H^1(\R^{n+1}_+,y^a)$ is the Gagliardo space $W^{\frac{1-a}{2},2}(\R^{n})
=H^s(\R^{n})$. The Sobolev embedding 
(see \cite{adams}) 
$$H^s(\R^{n}) \hookrightarrow L^{\frac{2n}{n-2s}}(\R^n)$$ 
(or into any $L^p(\R^{n})$ if $n=1\leq 2s$)
gives then the compactness of the inclusion  
$$
H_{u,a}(\Omega) \subset \subset L^2(\partial^0\Omega).
$$

Now, since $H_{u,a}(\Omega) \subset L^2(\partial^0\Omega)$ and 
$\tilde G$ has linear growth at infinity, 
it follows that  $\tilde E_{\Omega}$ is well 
defined, bounded below, and coercive in $H_{u,a}(\Omega)$. 
Hence, using the compactness
of the inclusion $H_{u,a}(\Omega)\subset \subset L^2(\partial^0\Omega)$,
taking a minimizing sequence in $H_{u,a}(\Omega)$
and a subsequence convergent in $L^2(\partial^0\Omega)$, 
we conclude that $\tilde E_{\Omega}$ admits an absolute
minimizer $w$ in $H_{u,a}(\Omega)$. 

Since $\tilde f$ is a continuous function, $\tilde E$ is a $C^1$
functional in $H_{u,a}(\Omega)$. 
Making first at second order variations of $\tilde E$
at the minimum $w$, we obtain that $w$ is a weak solution of~\eqref{probw} 
which satisfies~\eqref{stabw},
with $f$ and $f'$ replaced by $\tilde f$ and $\tilde f'$,
respectively, in both \eqref{probw} and~\eqref{stabw}.

Therefore, it only remains to show that the minimizer $w$
satisfies
$$
-1\le w \le 1 \quad\text{a.e. in } \Omega.
$$ 
We use that $-1$ and $1$ are, respectively, a subsolution
and a supersolution of \eqref{probw}, due to hypothesis \eqref{subsuper}.
We proceed as follows. We use that the first variation of
$\tilde E_\Omega$ at $w$ in the direction $(w-1)^+$ (the positive
part of $w-1$), is zero.  Since $|w|=|u|\le 1$ on $\partial^+\Omega$
and hence $(w-1)^+$ vanishes on $\partial^+\Omega$,
we have that $w+\epsilon (w-1)^+\in H_{u,a}(\Omega)$
for every~$\epsilon$. We deduce
\begin{equation*}
\begin{split}
0 & =\int_\Omega y^a \nabla w \, \nabla (w-1)^+ -\int_{\partial^0\Omega}
\tilde f(w) (w-1)^+ \\
& =\int_{\Omega\cap\{w\ge 1\}} y^a |\nabla (w-1)^+|^2 -
\int_{\partial^0\Omega\cap\{w\ge 1\}} f(1) (w-1)^+ \\
&\ge
\int_{\Omega} y^a |\nabla (w-1)^+|^2 ,
\end{split}
\end{equation*}
where we have used that $\tilde f(s) =f(1)$ for $s\ge 1$,
and that $f(1)\le 0$ by assumption.
We conclude that $(w-1)^+$ is constant, and hence identically zero.
Therefore, $w\le 1$ a.e. The inequality $w\ge -1$ is proved
in the same way, now using $f(-1)\ge 0$.
\hfill\qed

\section{Local minimality of layers and consequences.\\
Proof of Theorem~\ref{minimality}}

The fact that for reactions in the interior (that is, $s=1$ in our equation), 
layer solutions in $\R^n$ are necessarily local minimizers was found
by Alberti, Ambrosio, and one of the authors in \cite{AAC}. 
For the fractional case, this is the statement in Theorem~\ref{minimality} a) above. 
The proof in  \cite{AAC} also works in the fractional case, 
working with the extension problem. It uses two ingredients: the existence result from the previous
section (Lemma~\ref{existmin}) and the following uniqueness result
in the presence of a layer.

\begin{lem}
\label{uniqDN}
Assume that problem \eqref{extAlpha} admits a layer solution $u$. Then, for
every $R>0$, $u$ is the
unique weak solution of the problem 
\begin{equation}\label{prDNpm1}
\begin{cases}
L_a w=0&\text{ in } B^+_R\subset\R^{n+1}_+ \\ 
-1\le w\le 1  &\text{ in } B^+_R \\
(1+a)\dfrac{\partial w}{\partial\nu^a} =f(w)&\text{ on }\Gamma^0_R \\
w=u &\text{ on }\Gamma^+_R.
\end{cases}
\end{equation}
\end{lem}

\begin{proof}
We refer the reader to the proof of Lemma~3.1 in \cite{CSM} 
since the proof is identical in our case. 
Indeed, since the operator $L_a$ is invariant under translations in $x$, this allows 
to use the sliding method as in Lemma~3.1 of \cite{CSM} to get the uniqueness.
The only other important ingredient in the proof is the
Hopf boundary lemma; in our present context it can be found in Proposition 4.11 
and Corollary 4.12 of \cite{CS1}. 
\end{proof}

Part b) of Theorem~\ref{minimality} will follow from the
following proposition. It will  be useful also in other future arguments. 
Notice that the result for $n=1$ follows from our Modica estimate, Theorem 2.3
of \cite{CS1} (rewritten in Theorem~\ref{modthm} of the present paper). Instead, 
the following proof also works in higher dimensions
but only gives $G\ge G(L^-) = G(L^+)$ in $[-1,1]$ ---in contrast with the
strict inequality $G> G(-1) = G(1)$ obtained in dimension one from the Modica estimate
when $L^\pm=\pm 1$.

\begin{propo}\label{sameheight}
Let $u$ be a solution of \eqref{extAlpha} such that
$|u|<1$, and
\begin{equation*}
\label{limitsL2}
\lim_{x_1\to\pm\infty}u(x,0)=L^\pm  \quad\text{for every }
(x_2,\ldots ,x_{n})\in \R^{n-1},
\end{equation*}
for some constants $L^-$ and $L^+$ $($that could be equal$)$. 
Assume that $u$ is a local minimizer relative
to perturbations in $[-1,1]$. Then, 
\begin{equation*}\label{Gge}
G\ge G(L^-) = G(L^+)\quad \text{ in } [-1,1].
\end{equation*}
\end{propo}

\begin{proof}
It suffices to show that $G\ge G(L^-)$ and $G\ge G(L^+)$ in
$[-1,1]$. It then follows that $G(L^-)=G(L^+)$.
By symmetry, it is enough to establish that $G\ge G(L^+)$ in
$[-1,1]$. Note that this inequality, as well as the notion of
local minimizer, is independent of adding a constant to $G$.
Hence, we may assume that
$$
G(s)=0<G(L^+) \quad \text{ for some } s\in [-1,1],
$$
and we need to obtain a contradiction.
Since $G(L^+) > 0$, we have that 
$$
\frac{1}{1+a}G(t) \geq \varepsilon > 0 \quad\text{ for $t$ in
a neighborhood in $[-1,1]$ of } L^+
$$ 
for some $\varepsilon>0$.

Consider the points $(b,0,0)=(x_1=b,x_2=0,\ldots,x_n=0,y=0)$
on $\partial\R^{n+1}_+$. Since for $R>0$,
$$
E_{B^+_R(b,0,0)}(u) \geq \int_{\Gamma^0_R(b,0)}\frac{1}{1+a} G(u(x,0))\, dx
$$
and $u(x,0) \underset{x_1 \to + \infty}{\longrightarrow} L^+$, we
deduce
\begin{equation} \label{lowbou}
\varliminf_{b\to +\infty} E_{B^+_R(b,0,0)}(u) \geq c(n)\, \varepsilon R^{n}
\qquad \text {for all }  R > 0.
\end{equation}
The constant $c(n)$ depends only on $n$.

The lower bound \eqref{lowbou} 
will be a contradiction with an upper bound for the energy of $u$,
that we obtain using the local minimality of $u$.

For $R>1$, let $\xi_R$ be a smooth function in $\RR^{n+1}$ such that
$0\leq \xi_R\leq 1$,
$$
\xi_R= 
\begin{cases} \displaystyle 
1 & \text{ in} \,\,\,B^+_{(1-\eta)R} \\
0 & \text{ on} \,\,\,\RR^{n+1}_+ \backslash B_R^+ ,
\end{cases}
$$
and $|\nabla\xi_R|\leq C(n)(\eta R)^{-1}$,
where $\eta \in (0,1)$ is to be chosen later.
Let 
$$
\xi_{R,b}(x,y):=\xi_R(x_1+b,x_2,\ldots,x_n,y).
$$
Since
$$
(1-\xi_{R,b})u+\xi_{R,b}s=u+\xi_{R,b}(s-u)
$$
takes values in $[-1,1]$ and agrees with $u$ on $\Gamma^+_R(b,0,0)$,
we have that 
$$
E_{B^+_R(b,0,0)}(u) \leq E_{B^+_R(b,0,0)}(u+\xi_{R,b}(s-u)).  
$$

Next, we bound by above this last energy. Since $G(s)=0$, the potential energy is 
only nonzero in $B^+_{R}\setminus B^+_{(1-\eta)R}$, which has measure
bounded above by $C(n)\eta R^{n}$. On the other hand, since we proved
in Lemma~4.8(i) of \cite{CS1} that
$$
\|\nabla_x u\|_{L^\infty(B^+_R(x,0))} \rightarrow 0\,\,\mbox{as $x_1
  \rightarrow \pm \infty$},
$$
we deduce that
\begin{eqnarray*}
& & \hspace{-2cm}\varlimsup_{b\to +\infty} \int_{B^+_R(b,0,0)} y^a  
|\nabla\{u+\xi_{R,b}(s-u)\}|^2  \leq 
2\int_{B_R^+} y^a |\nabla\xi_R|^2 \\
&\leq&  \frac{C(n)}{\eta^2 R^2} R^n \int_0^R y^a\,dy
=  C(n)\frac{R^{n+1+a}}{\eta^2 R^2} =  C(n)\frac{R^{n-2s}}{\eta^2}.
\end{eqnarray*}
Putting together the bounds for Dirichlet and potential energies, we conclude that 
\begin{eqnarray*}
\varlimsup_{b\to +\infty} E_{B^+_R(b,0,0)}(u) &\leq&  
\varlimsup_{b\to +\infty}E_{B^+_R(b,0,0)}(u+\xi_{R,b}(s-u))\\
&\leq & C\{ \eta  R^{n} + \eta^{-2}R^{n-2s}\},
\end{eqnarray*}
for some constant $C>0$ depending only on $n$, $a$, and $G$. 

Recalling the lower bound \eqref{lowbou}, we now choose $\eta$ small enough so that
$C\eta=(1/2)c(n)\varepsilon$. In this way,  \eqref{lowbou} and the last upper bound lead to
$(1/2)c(n)\varepsilon R^n\leq C \eta^{-2}R^{n-2s}$. This is a contradiction
when $R$ is large enough. 
\end{proof}

\medskip
\noindent
{\it Proof of Theorem}~\ref{minimality}.
We proceed exactly as in the proof of Theorem~1.4 in \cite{CSM}, page 1708.

To prove part a), for $R>1$ we consider problem \eqref{prDNpm1} in a half-ball. 
Lemma \ref{existmin} gives the existence of a minimizer $w$ with $-1\leq w\leq 1$.
Note that in the lemma one needs condition \eqref{subsuper}. But in the presence of a layer,
we showed in Lemma 4.8(i) of \cite{CS1} that one has $f(-1)=f(1)=0$.

On the other hand, 
Lemma~\ref{uniqDN} states that the layer $u$ is the unique solution of
\eqref{prDNpm1} .
Thus, $u\equiv w$ in $B_R^+$. This shows that $u$ is a local minimizer.

To prove part b),
$G'(-1)=G'(1)=0$ was shown in Lemma 4.8(i) of \cite{CS1}.
We have established
the other relation, $G\geq G(-1)=G(1)$ in $[-1,1]$, in Proposition~\ref{sameheight}
above.
\hfill\qed

\section{Monotonicity and 1D symmetry of stable solutions in $\RR^2$.
Proof of Theorem~\ref{symNL}}

To prove Theorem~\ref{symNL}, we need two lemmas.
The following one, applied with $d(x)=-(1+a)^{-1}f'(u(x,0))$,
establishes an alternative 
criterium for a solution $u$ of~\eqref{extAlpha} to be stable.

\begin{lem}\label{polipo}
Let $d$ be a bounded and H\"older continuous 
function on $\partial\R^{n+1}_+$. Then,
\begin{equation}\label{stablebis}
\int_{\R^{n+1}_+} y^a |\nabla\xi|^2 +\int_{\partial\R^{n+1}_+}
d(x)\xi^2\ge 0
\end{equation}
for every function $\xi\in C^1(\overline{\R^{n+1}_+})$ with compact support in 
$\overline{\R^{n+1}_+}$, if and only if 
there exists a H\"older continuous function $\varphi$ in 
$\overline{\R^{n+1}_+}$ such that $\varphi >0$ in 
$\overline{\R^{n+1}_+}$, $\varphi\in H^1_{\rm loc}(\overline{\R^{n+1}_+},y^a)$, and
\begin{equation} \label{linear}
\begin{cases}
L_a \varphi = 0&\text{ in } \R^{n+1}_+\\ 
\dfrac{\partial\varphi}{\partial\nu^a}+d(x)\varphi = 0&\text{ on }
\partial\R^{n+1}_+ .
\end{cases}
\end{equation}
\end{lem}

\noindent
{\it Proof.} 
First, assume the existence of a positive solution
$\varphi$ of \eqref{linear}, as in the statement of the
lemma. Let $\xi\in C^1(\overline{\R^{n+1}_+})$
have compact support in $\overline{\R^{n+1}_+}$.
We multiply $L_a \varphi=0$ by $\xi^2/\varphi$, integrate by
parts and use Cauchy-Schwarz inequality to obtain~\eqref{stablebis}.

For the other implication, we follow \cite{CSM}. 
Assume that \eqref{stablebis}
holds for every  $\xi\in C^1(\overline{\R^{n+1}_+})$
with compact support in~$\overline{\R^{n+1}_+}$.
For every $R>0$, let $\lambda_R$
be the infimum of the quadratic form
\begin{equation}\label{quad}
Q_R(\xi)= \int_{B^+_R} y^{a}|\nabla\xi|^2 +\int_{\Gamma^0_R}
d(x)\xi^2
\end{equation}
among functions in the class $S_R$, defined by
\begin{eqnarray*}
& S_R=  \left\{ \xi\in H^1(B_R^+,y^a )\, : \, \xi\equiv 0 \text{ on } \Gamma^+_R
\text{ and } \int_{\Gamma^0_R} \xi^2 =1\right\}\\
&  \subset H_{0,a}(B_R^+)=\left\{ \xi\in H^1(B_R^+,y^a)\, : \, \xi\equiv 0 
\text{ on } \Gamma^+_R \right\}.
\end{eqnarray*}
We recall that the space $H_{0,a}(B_R^+)$ was already defined in \eqref{defaffine}.

By our assumption, $\lambda_R\ge 0$ for every $R$. 
By definition it is clear 
that $\lambda_R$ is a nonincreasing function of $R$.
Next, we show that $\lambda_R$ is indeed a decreasing
function of~$R$. As a consequence, we deduce that
$\lambda_R >0$ for every $R$, and this will be important in the sequel.

To show that $\lambda_R$ is decreasing in $R$,
note first that since $d$ is assumed to be a bounded function, 
the functional $Q_R$ is bounded below in the class $S_R$. For the
same reason, any minimizing sequence $(\xi_k)$ has
$(\nabla\xi_k)$ uniformly bounded in $L^2(B_R^+,y^a)$. Hence, by the
compact inclusion $H_{0,a} (B_R^+) \subset\subset
L^2(\Gamma_R^0)$ (already mentioned in the proof of Lemma~\ref{existmin}), 
we conclude that the infimum of
$Q_R$ in $S_R$ is achieved by a function $\phi_R\in S_R$.

Moreover,
we may take $\phi_R\ge 0$, since $|\phi|$ is a
minimizer whenever $\phi$ is a minimizer.
Note that $\phi_R \ge 0$ is a solution, not identically zero, of
\begin{equation*} 
\begin{cases}
L_a \phi_R = 0 & \text{ in } B^+_R\\ 
\dfrac{\partial\phi_R}{\partial\nu^a}+d(x)\phi_R = 
\lambda_R\phi_R & \text{ on } \Gamma^0_R \\
\phi_R = 0 & \text{ on } \Gamma^+_R .
\end{cases}
\end{equation*}
It follows from the strong maximum principle that $\phi_R >0$
in $B_R^+$.

We can now easily prove that $\lambda_R$ is decreasing in $R$.
Indeed, arguing by contradiction, assume that $R_1<R_2$ and 
$\lambda_{R_1}=\lambda_{R_2}$. Multiply $L_a \phi_{R_1}=0$ 
by $\phi_{R_2}$, integrate by parts, use the equalities
satisfied by $\phi_{R_1}$ and $\phi_{R_2}$,
and also the assumption $\lambda_{R_1}=\lambda_{R_2}$. We obtain
$$
\int_{\Gamma^+_{R_1}} \frac{\partial\phi_{R_1}}{\partial\nu^a}
\phi_{R_2} =0,
$$
and this is a contradiction since, on $\Gamma^+_{R_1}$, we have 
$\phi_{R_2}>0$ and
the derivative $\partial\phi_{R_1}/\partial\nu^a <0$.

Next, using that $\lambda_R>0$ we obtain
$$
\int_{B^+_R}y^a |\nabla\xi|^2 +\int_{\Gamma^0_R} d(x)\xi^2 \ge
\lambda_R \int_{\Gamma^0_R} \xi^2 \ge -\delta_R \int_{\Gamma^0_R} 
d(x)\xi^2,
$$
for all $\xi\in H_{0,a}(B_R^+)$, where $\delta_R$ is taken
such that  $0< \delta_R \le 
\lambda_R/\Vert d\Vert_{L^\infty}$. From the last inequality,
we deduce that
\begin{equation}\label{coercfirst}
\int_{B^+_R} y^a |\nabla\xi|^2 +\int_{\Gamma^0_R} d(x)\xi^2 \ge
\epsilon_R \int_{B^+_R} y^a |\nabla\xi|^2
\end{equation}
for all $\xi\in H_{0,a}(B_R^+)$, for $\epsilon_R >0$
given by $\epsilon_R =1-1/(1+\delta_R)$.

It is now easy to prove that, 
for every constant $c_R>0$, there
exists a solution $\varphi_R$~of
\begin{equation}\label{prvarR}
\begin{cases}
L_a \varphi_R = 0 & \text{ in }B_R^+ \\
\dfrac{\partial\varphi_R}{\partial\nu^a}+d(x)\varphi_R = 0 &\text{ on }
\Gamma^0_R \\
\varphi_R  =c_R  & \text{ on }\Gamma^+_R .
\end{cases}
\end{equation}
Indeed, rewriting this problem for the function
$\psi_R=\varphi_R-c_R$, we need to solve
\begin{equation*}
\begin{cases}
L_a \psi_R = 0 & \text{ in }B_R^+ \\
\dfrac{\partial\psi_R}{\partial\nu^a}+d(x)\psi_R +c_R d(x) =0 &\text{ on }
\Gamma^0_R \\
\psi_R  =0  & \text{ on }\Gamma^+_R .
\end{cases}
\end{equation*}
This problem can be solved by minimizing the functional
$$
\int_{B^+_R} \frac{y^a}{2}|\nabla\xi|^2 +
\int_{\Gamma^0_R} \left\{ \frac{1}{2} d(x)\xi^2 +
c_Rd(x)\xi\right\}
$$
in the space $H_{0,a}(B_R^+)$.
Note that the functional is bounded below and coercive,
thanks to inequality \eqref{coercfirst}. Finally, 
the compact inclusion $H_{0,a}(B_R^+)\subset\subset L^2(\Gamma_R^0)$
gives the existence of a minimizer. 

Next, we claim that
\begin{equation*}\label{posphiR}
\varphi_R>0 \quad\text{ in } \overline{B_R^+}.
\end{equation*}
Indeed, the negative part $\varphi^-_R$ of $\varphi_R$
vanishes on $\Gamma^+_R$. Using this, \eqref{prvarR}, and the definition
\eqref{quad} of $Q_R$, it is easy to verify that $Q_R(\varphi^-_R)=0$. 
By definition of the first eigenvalue
$\lambda_R$ and the fact that $\lambda_R >0$, this implies that 
$\varphi_R^- \equiv 0$, i.e., $\varphi_R\ge 0$. Now, the Hopf's
maximum principle (Corollary 4.12 of \cite{CS1}) 
gives $\varphi_R>0$ up to the boundary.

Finally, we choose the constant $c_R>0$ in
\eqref{prvarR} to have $\varphi_R(0,0)=1$. 
Then, by the Harnack inequality in Lemma 4.9 of \cite{CS1}
applied to $\varphi_S$ with $S>4R$,  we deduce
$$
\sup_{B_R^+} \varphi_S \le C_R \quad\text{ for all }S>4R.
$$
Now that $(\varphi_S)$ is uniformly bounded in $B_R^+$, we use
\eqref{reg11} in Lemma~\ref{regularity1} 
to get a uniform $C^\beta(\overline{B_{R/2}^+})$ bound for the
sequence. Note that the constant $C_R^1$ in \eqref{reg11} depends on
the $L^\infty$ (and not on the $C^\sigma$) of $d\varphi_S$, which we
already controlled. However, to apply  Lemma~\ref{regularity1} we need to know
that $d\varphi_S$ is $C^\sigma$. This is a consequence of the linear problem solved
by $\varphi_S$ and the fact that $d\varphi_S\in L^\infty$. This leads to
$\varphi_S\in C^\sigma$ as shown in the beginning of the proof
of Lemma~4.5 of \cite{CS1}.

Now, the uniform $C^\beta(\overline{B_{R/2}^+})$ bound gives that
a subsequence of $(\varphi_S)$
converges locally in $\overline{\R^{n+1}_+}$ to a 
$C^\beta_{\rm loc}(\overline{\R^{n+1}_+})$ solution $\varphi>0$ of \eqref{linear}.
\hfill\qed
\medskip

The previous lemma
provides a direct proof of the fact that every 
layer solution $u$ of \eqref{problem} is stable, which was 
already known by the local minimality property established in section~5. 
Indeed, we simply note that 
$\varphi =u_{x_1}$ is strictly
positive and solves the linearized 
problem \eqref{linear}, with $d(x)=-(1+a)^{-1}f'(u(x,0))$. 
Hence, the stability of $u$ 
follows from Lemma~\ref{polipo}.

We use now the previous lemma to establish
a result that leads easily to the monotonicity
and the 1D symmetry of stable solutions in dimensions $n=1$ and $n=2$,
respectively.

\begin{lem}
\label{partials}
Assume that $n\le 2$ and that $u$ is a bounded stable solution of
\eqref{extAlpha}. Then, there exists a H\"older continuous
function $\varphi >0$  in $\overline{\R^{n+1}_+}$ such that,
for every $i=1,\dots,n$,
$$
u_{x_i} =c_i \varphi \quad \text{in }\R^{n+1}_+
$$
for some constant $c_i$.
\end{lem}

\begin{proof}
Since $u$ is assumed to be a stable solution, then \eqref{stablebis}
holds with $d(x):=-(1+a)^{-1}f'(u(x,0))$. 
Note that $d\in C^\beta$ by Lemma~\ref{regularity1}.
Hence, by Lemma~\ref{polipo}, there exists a H\"older continuous function
$\varphi >0$ in $\overline{\R^{n+1}_+}$ such that
\begin{equation*} \label{linearf}
\begin{cases}
L_a \varphi = 0&\text{ in } \R^{n+1}_+\\ 
\dfrac{\partial\varphi}{\partial\nu^a }-(1+a)^{-1}f'(u(x,0))\varphi = 0&\text{ on }
\partial\R^{n+1}_+ .
\end{cases}
\end{equation*}

For $i=1,\ldots,n$ fixed, consider the function
$$
\sigma =\frac{u_{x_i}}{\varphi}.
$$
The goal is to 
prove that $\sigma$ is constant in $\R^{n+1}_+$.

Note first that 
$$
\varphi^2 \nabla \sigma =\varphi \nabla u_{x_i}-
u_{x_i}\nabla \varphi.
$$
Thus, we have that
$$
\textrm{div\,}(y^a \varphi^2 \nabla\sigma)=0 \quad\; \text{in }\R^{n+1}_+ .
$$
Moreover, we have $\frac{\partial \sigma}{\partial \nu^a}=0$ on
$\partial\R^{n+1}_+$ since
$$
\varphi^2 \sigma_y =\varphi u_{y x_i}-
u_{x_i}\varphi_{y} =0
$$
due to the fact that $u_{x_i}$ and
$\varphi$ both satisfy the same linearized boundary condition.

We can use the Liouville property that we established in \cite{CS1}
(Theorem~4.10 of \cite{CS1}),
and deduce that $\sigma$ is constant, provided
that the growth condition
\begin{equation}\label{quadratic}
\int_{B_R^+} y^a (\varphi\sigma)^2  \leq CR^2
\quad\quad\text{ for all } R>1 
\end{equation}
holds for some constant $C$ independent of $R$. 
But note that
$\varphi\sigma=u_{x_i}$, and therefore
$$
\int_{B_R^+} y^a (\varphi\sigma)^2 \le
\int_{B_R^+} y^a |\nabla u|^2 .
$$
Thus, we need to estimate this last quantity.

To do this, we perform a simple energy estimate. Multiply the equation
$\textrm{div\,}(y^a\nabla u)=0$ by
$\xi^2 u$ and integrate in $B_{2R}^+$, where $0\leq \xi\leq 1$ is a $C^\infty$ cutoff
function with compact support in $B_{2R}$ such that $\xi\equiv 1$ 
in $B_{R}$ and $|\nabla \xi|\le 2/R$. We obtain 
\begin{eqnarray*}
&&  \int_{B_{2R}^+}
y^a\, \{  \xi ^2|\nabla u |^2+2 \xi u \nabla \xi \cdot \nabla u \}=
\int_{\Gamma^0_{2R}} (1+a)^{-1}f(u) \xi^2u. 
\end{eqnarray*} 
Thus, by Cauchy-Schwarz inequality and since $u$ and $\xi$ are bounded,
$$
\int_{B_{2R}^+}y^a\,\xi^2|\nabla u|^2
\le \frac 12 \int_{B_{2R}^+}y^a\,\xi^2|\nabla u|^2
+C \int_{B_{2R}^+} y^a |\nabla \xi|^2
+C |\Gamma^0_{2R}|
$$
for a  constant $C$ independent of $R$. Absorbing the first term on the left hand side,
using that $\xi\equiv 1$ in $B_R$ and $|\nabla \xi|\le 2/R$,
and computing $\int_0^{2R} y^a\, dy$, we deduce
\begin{equation*}
\int_{B_{R}^+}y^a\,|\nabla u|^2 \leq C \{R^{-2}R^n R^{1+a}+R^n\}
=C \{R^{n-2s}+R^n\} \leq C R^2
\end{equation*}
since $n\leq 2$. This establishes \eqref{quadratic} and finishes the proof.
\end{proof}

We can now give the

\medskip
\noindent
{\it Proof of Theorem}~\ref{symNL}.
Let $n=2$. The extension $u$ of $v$ is a bounded stable solution of \eqref{extAlpha}
with $f$ replaced by $(1+a)d_s^{-1}f$. 

Lemma~\ref{partials} establishes that
$u_{x_i}\equiv c_i \varphi$ for some constants $c_i$, for $i=1,2$.
If $c_1=c_2=0$, then $u$ is constant. 
Otherwise we have that
$c_2u_{x_1}-c_1u_{x_2}\equiv 0$ and we conclude that
$u$ depends only on
$y$ and on the variable parallel to $(0,c_1,c_2)$.
That is,
$$
u(x_1,x_2,y)=u_0\left( c_1x_1+c_2x_2)/(c_1^2+c_2^2)^{1/2},y\right) =
u_0(z,y),
$$
where $z$ denotes the variable parallel
to $(0,c_1,c_2)$. We have that $u_0$ is a solution of the same nonlinear
problem now for $n=1$ thanks to the extension characterization; recall
that the constant $d_s$ in \eqref{cttNeumann} does not depend on the dimension.

In particular
$\partial_xu_0 = (c_1^2+c_2^2)^{1/2}\varphi$,
and hence $\partial_xu_0 >0$ everywhere. 
This finishes the proof of the theorem.
\hfill\qed

\section{Layer solutions in $\RR$}

This section is devoted to the case $n=1$. The Modica estimate 
that we proved in \cite{CS1} (see Theorems~\ref{modthm} and \ref{necLayers} above) gave that 
$$G>G(-1)=G(1)\,\text{ in }(-1,1)$$ 
is a  necessary condition for the existence of a layer solution in $\RR$.
Note the strict inequality in $G>G(\pm 1)$. 

The rest of the section is dedicated to prove
the existence of a layer solution under the above condition on $G$, 
in addition to $G'(-1)=G'(1)=0$, as stated 
in Theorem~\ref{existNonlocal}.
The existence part of Theorem~\ref{existNonlocal} is all contained in the following lemma.

\begin{lem}\label{keyexist}
Assume that $n=1$, and that
$$
G'(-1) = G'(1) = 0 \quad\text{ and }\quad G> G(-1)=G(1)\ \text{ in }
(-1,1).
$$
Then, for every $R>0$, there exists a function
$u_R\in C^\beta(\overline{B_R^+})$ for some $\beta \in (0,1)$ independent of $R$, such that 
$$
-1<u_R< 1 \quad\text{in }\overline{B_R^+},
$$ 
$$
u_R(0,0) = 0,
$$  
$$
\partial_x u_R \ge 0\quad \text{in } B_R^+ ,
$$
and $u_R$ is a minimizer of the energy in $B_R^+$, in the sense that
$$
E_{B_R^+}(u_R)\le E_{B_R^+}(u_R+\psi)
$$
for every $\psi\in C^1(\overline{B_R^+})$ with compact support in 
$B_R^+\cup\Gamma_R^0$ and such that $-1\le u_R+\psi\le 1$ in $B_R^+$. 

Moreover, as a consequence of the previous statements, 
we will deduce that a subsequence 
of $(u_R)$ converges in $C^\beta_{\rm loc}(\overline{\R^2_+})$
to a layer solution $u$  
of \eqref{extAlpha}.
\end{lem}

\begin{proof}
For $R >1$, let
$$
Q_R^+ = (-R,R) \times (0,R^{1/8}).
$$
Consider the function 
$$
v^R(x,y) =v^R(x)= 
\frac{\arctan x}{\arctan R} \quad \text{ for } (x,y)\in 
\overline{Q_R^+}.
$$
Note that $-1\le v^R \le 1$ in $Q_R^+$.

Let $u^R$ be an absolute minimizer of $E_{Q_R^+}$ in the
set of functions $v \in H^1(Q_R^+,y^a)$ such that $|v|\le 1$ 
in $Q_R^+$ and $v \equiv v^R$ on $\partial^+Q_R^+$ in the weak sense.
Since we are assuming $G'(-1)=G'(1)=0$, the existence of such minimizer
was proved in Lemma \ref{existmin}. We have that
$u^R$ is a weak solution of
\begin{equation*}
\begin{cases}
L_a u^R=0&\text{ in } Q_R^+\\ 
(1+a)\dfrac{\partial u^R}{\partial\nu^a} =f(u^R)&\text{ on }\partial^0 Q_R^+\\
u^R=v^R &\text{ on }\partial^+ Q_R^+,
\end{cases}
\end{equation*}
and, by the strong maximum principle and the Hopf's lemma (Corollary~4.12 of \cite{CS1}),
$$
\abs{u^R} < 1 \quad \text{in } \overline{Q_R^+}.
$$
The function $u^R$ is H\"older continuous by Lemma~\ref{regularity1}.

We follow the method developed in \cite{CSM} and proceed in three steps.
First we show:
\begin{equation}\label{enerur1}
\text{Claim 1:} \quad E_{Q_R^+}(u^R) \leq C R^{1/4}
\end{equation}
for some constant $C$ independent of $R$. Here we take
$G-G(-1)=G-G(1)$ as boundary energy potential.  We will
use this energy bound to prove in a second step that,
for $R$ large enough, 
\begin{equation}\label{sets1}
\text{Claim 2:} \quad\abs{\{u^R(\cdot,0) > 1/2\}} \geq R^{3/4} 
\text{ and } \abs{\{u^R(\cdot,0) <-1/2\}} \geq R^{3/4}.
\end{equation}
Finally, in a third step independent of the two previous ones, 
we prove that
\begin{equation}
\label{lastmon}
\text{Claim 3:} \quad
u_x^R = \partial_x u^R \ge 0 \quad \text{ in } Q_R^+.
\end{equation}

With the above three claims, we can easily finish the proof of the lemma,
as follows.
Since $u^R(\cdot,0)$ is nondecreasing (here, this is a key point)
and continuous in $(-R,R)$, we
deduce from \eqref{sets1} that for $R$ large enough,
\begin{equation*}
u^R(x_R,0) = 0 \quad \text{for some } x_R \text{ such that } \abs{x_R} \le
R - R^{3/4} .
\end{equation*}
Since $\abs{x_R} \leq R - R^{3/4} < R - R^{1/8}$, we have that
$$
\overline{B^+_{R^{1/8}}}(x_R,0) \subset (-R,R) \times [0,R^{1/8}]
\subset \overline{Q_R^+}.
$$
We slide $u^R$ and define
$$
u_{R^{1/8}}(x,y) = u^R(x+x_R,y) \quad \text{for}\quad (x,y)\in 
\overline{B^+_{R^{1/8}}}(0,0).
$$
Then, relabeling the index by setting $S=R^{1/8}$, we have that $u_S\in
C^\beta(\overline{B^+_S}(0,0))$,
$-1<u_S < 1$ in $\overline{B^+_S}(0,0)$, 
$u_S(0,0) = 0$, and $\partial_x u_S \ge 0$ in
$B_S^+(0,0)$. Moreover, $u_S$ is a minimizer in $B^+_S(0,0)$ 
in the sense of Lemma~\ref{keyexist}. This follows from extending
a given $H^1$ function $\psi$ with compact support in
$(B_S^+\cup \Gamma^0_S)(x_R,0)$, and with $|u+\psi|\le 1$ in 
$B_S^+(x_R,0)$,
by zero in $Q_R^+\setminus B_S^+(x_R,0)$. Hence $\psi$ is a
$H^1(Q_R^+)$ function. Then one uses the minimality of $u^R$
in $Q_R^+$ and the fact that the energies of $u^R$ and $u^R+\psi$ 
coincide in $Q_R^+\setminus B_S^+(x_R,0)$ to deduce the desired
relation between the energies in $B_S^+(x_R,0)$.

Now we prove the last statement of the lemma: a subsequence of $(u_R)$
converges to a layer solution. Note that we use the sequence $(u_R)$ 
just constructed, and not the sequence $(u^R)$ in the beginning of the
proof.

Let $S>0$. Since $\abs{u_R} < 1$, Lemma~\ref{regularity1} gives 
$C^\beta(\overline{B_S^+})$ estimates for $u_R$, uniform
for $R\geq 2S$.
Hence, for a subsequence (that we still denote by $u_R$), we
have that $u_R$ converges locally uniformly as $R\to \infty$ 
to some function
$u \in C^\beta_{\rm loc}(\overline{\R^2_+})$. By the additional bound \eqref{reg12} on 
$y^a u_y$ given by Lemma~\ref{regularity1}, one can pass to the limit in the 
weak formulation and $u$ weakly solves \eqref{extAlpha}

We also have that $|u|\le 1$,
$$
u(0,0) =0 \quad \text{ and }\quad u_x \geq 0 \text{ in }\R^2_+.
$$
Since $u(0,0)=0$, we have $|u|\not\equiv 1$ and hence $|u|<1$ in $\overline{\R^2_+}$, 
by the strong maximum principle and Hopf's lemma.
Note that $\pm 1$ are solutions of the problem since, by hypothesis,
$G'(\pm 1)=f(\pm 1)=0$.

Let us now show that $u$ is a local minimizer
relative to perturbations in $[-1,1]$.
Indeed, let $S>0$ and 
$\psi$ be a $C^1$ function with compact support in
$B_S^+\cup \Gamma^0_S$ and such that $|u+\psi|\le 1$ in $B_S^+$. 
Extend $\psi$ to
be identically zero outside $B_S^+$, so that $\psi\in
H^1_{\rm loc}(\overline{\R^2_+})$.
Note that, since $-1<u<1$ and $-1\le u+\psi\le 1$, we have 
$-1<u+(1-\epsilon)\psi <1$ in $\overline{B_S^+}$ for every $0<\epsilon<1$.
Hence, by the local convergence of $(u_R)$ towards $u$, 
for $R$ large enough we have $B_S^+\subset B_R^+$ and
$-1 \le u_R+(1-\epsilon)\psi \le 1$ in $B_S^+$, 
and hence also in $B_R^+$. 
Then, since $u_R$ is a minimizer in $B_R^+$, we have
$E_{B_R^+}(u_R) \leq E_{B_R^+}(u_R + (1-\epsilon)\psi)$ 
for $R$ large. Since $\psi$ has support in
$B_S^+\cup \Gamma^0_S$, this is equivalent to 
$$
E_{B_S^+}(u_R) \leq E_{B_S^+}(u_R + (1-\epsilon)\psi) 
\quad \text{for $R$ large}.
$$
Letting $R \to \infty$, we deduce $E_{B_S^+}(u) \leq
E_{B_S^+}(u +(1-\epsilon) \psi)$. We conclude now
by letting $\epsilon\to 0$.

Finally, since $u_x\ge 0$, the limits 
$L^{\pm}=\lim_{x\to\pm\infty}u(x,0)$ exist. To establish that
$u$ is a layer solution, it remains only to prove
that $L^\pm=\pm 1$. For this, note that we can 
apply Proposition~\ref{sameheight} to  $u$, a local minimizer
relative to perturbations in $[-1,1]$,
and deduce that
$$
G\ge G(L^-)=G(L^+) \quad\text {in } [-1,1].
$$
Since in addition $G>G(-1)=G(1)$ in $(-1,1)$ by hypothesis, we infer that $|L^\pm|=1$. 
But $u(0,0)=0$ and thus $u$ cannot be identically $1$ or $-1$.
We conclude that $L^-=-1$ and $L^+=1$, and therefore
$u$ is a layer solution.

We now go back to the functions $u^R$ defined in the beginning of the
proof, and proceed to establish the three claims made above.

\smallskip
{\it Step}~1. Here we prove \eqref{enerur1}
for some constant $C$ independent of $R$. We take
$G-G(-1)=G-G(1)$ as boundary energy potential. 

Since $E_{Q_R^+}(u^R) \leq E_{Q_R^+} (v^R)$, we simply 
need to bound the energy of $v^R$.
We have 
$$
\abs{\nabla v^R} = \abs{\partial_x v^R} =
\frac1{\arctan R}\  \frac1{1+x^2} \leq C \frac1{1+x^2},
$$ 
and hence
$$
\int_{Q_R^+} y^a \abs{\nabla v^R}^2 \leq CR^{\frac{1+a}{8}} \int_{-R}^R
\frac{dx}{(1+x^2)^2} \leq CR^{1/4} 
$$
since $0 < 1+a < 2$.

Next, since $G\in C^{2,\gamma}$, $G'(-1) = G'(1) = 0$ and $G(-1) = G(1)$, we have that 
$$
G(s) -G(1) \leq C(1 + \cos(\pi s)) \quad\text {for all }s\in
[-1,1],
$$ 
for some constant $C>0$.
Therefore, using that $\pi/\arctan R>2$, we have 
\begin{equation*}
\begin{split}
G(v^R(x,0))-G(1) 
&\le C\left\{ 1+\cos \Big( \pi \frac{\arctan x}{\arctan R}\Big)\right\} \\
&\le C\big( 1+\cos (2\arctan x)\big) \\
&= C2 \cos^2 (\arctan x) = \frac{2C}{1+x^2}.
\end{split}
\end{equation*}
We conclude that
$$
\int_{-R}^R \{G(v^R(x,0))-G(1)\}\, dx \leq C
\int_{-R}^R\frac{dx}{1+x^2} \leq C.
$$
This, together with the above bound for the Dirichlet energy, 
proves \eqref{enerur1}.

\smallskip
{\it Step}~2. Here we prove \eqref{sets1} for $R$ large enough.

Since $u^R \equiv v^R$ on $\{y = R^{1/8}\}$ and $\int_{-R}^R 
v^R(x) \, dx = 0$, we have
\begin{equation*}
\int_{-R}^R u^R(x,0) \, dx  = \int_{-R}^R u^R (x,0) \, dx - \int_{-R}^R
u^R(x,R^{1/8}) \, dx = -\int_{Q_R^+} u_y^R.
\end{equation*}
The energy bound \eqref{enerur1} and the
hypothesis that
$G-G(1) \geq 0$ give that the Dirichlet energy alone also satisfies the bound in \eqref{enerur1}.
We use this together with the previous equality and Cauchy-Schwarz inequality (writing $|u^R_y|= 
y^{-a/2} y^{a/2} |u^R_y|$), to deduce
\begin{equation} \label{intur}
\begin{split}
\Big|\int_{-R}^R u^R(x,0) dx\Big| & \leq \int_{Q_R^+} |u_y^R| 
   \leq \bigg\{ \int_{Q_R^+} y^{-a}\cdot \int_{Q_R^+} y^a \abs{\nabla
    u^R}^2\bigg\}^{1/2} \\
  & \leq C \Big\{ RR^{(1-a)/8}R^{1/4}\Big\}^{1/2}  \leq CR^{3/4},
\end{split}
\end{equation}
since $0 < 1-a < 2$.

Next, by \eqref{enerur1} we know that 
$\int_{-R}^R \{ G(u^R(x,0))-G(1)\} \, dx\le CR^{1/4}\leq
CR^{3/4}$. On the other hand,
$G(s)-G(1) \geq \varepsilon > 0$ if $s \in [-1/2, 1/2]$, for some
$\varepsilon > 0$ independent of $R$. Moreover, $G-G(1) \geq 0$ 
in $(-1,1)$. We deduce
$$
\varepsilon \big|\{\abs{u^R(\cdot,0)} \leq 1/2\}\big| \leq \int_{-R}^R 
\{ G(u^R(x,0))-G(1) \}\, dx \leq CR^{3/4},
$$
and therefore  $\big|\{\abs{u^R(\cdot,0)} \leq 1/2\}\big| \leq  CR^{3/4}$.
This combined with \eqref{intur} leads to
\begin{equation} \label{secintur}
\Big| \int_{(-R,R) \cap \{\abs{u^R(\cdot,0)} > 1/2\}} 
u^R(x,0)\, dx \Big|  \leq  CR^{3/4}.
\end{equation}

We claim that 
$$
\abs{\{u^R(\cdot,0) > 1/2\}} \geq R^{3/4} \quad\text{ for $R$ large
enough}.
$$
Suppose not.
Then, using \eqref{secintur} and  $\abs{\{u^R(\cdot,0) > 1/2\}} \leq R^{3/4}$,
we obtain
$$
\displaystyle \frac12\abs{\{u^R(\cdot,0) < -1/2\}} 
\leq  \Big| \displaystyle \int_{(-R,R)\cap \{u^R(\cdot,0) < -1/2\}} 
 u^R(x,0)\, dx \Big|\le CR^{3/4}.
$$
Hence, all the three sets $\{|u^R(\cdot,0)| \leq 1/2\}$, $\{u^R(\cdot,0) >
1/2\}$, and $\{u^R(\cdot,0) < -1/2\}$ would have length smaller than
$CR^{3/4}$. This is a contradiction for $R$ large, since these
sets fill $(-R,R)$.

\smallskip
{\it Step}~3. Here we establish the monotonicity result \eqref{lastmon}.
This is done exactly as in Step~3 in the proof in \cite{CSM}, to which we refer.
One simply uses the sliding method with the aid of the Hopf boundary lemma of \cite{CS1}. 
\end{proof}

\medskip
\noindent
{\it Proof of Theorem}~\ref{existNonlocal}.
The necessary conditions on $G$ follow from our previous paper \cite{CS1};
see Theorem~\ref{necLayers} above.

That the conditions are sufficient for the existence of a layer $v=v(x)$
follows from Lemma~\ref{keyexist}, which gives a layer solution $u=u(x,y)$
of the corresponding nonlinear extension problem \eqref{extAlpha} and then by taking
$v:=u(\cdot,0)$. Note that we consider the extension problem with $f$ 
replaced by $(1+a)d_s^{-1}f$ due to the relation \eqref{cttNeumann} between the
fractional Laplacian and the Neumann derivative.

Finally, the proof of the uniqueness result follows exactly that of 
Lemma~5.2 in \cite{CSM} for the half-Laplacian. It uses the sliding method
combined with the maximum principle Lemma 4.13 and Remark 4.14 in our previous paper \cite{CS1}.  
\hfill\qed

\medskip
\noindent
{\it Proof of Theorem}~\ref{classif}.
The proof is identical to that of Proposition~6.1 in \cite{CSM}, page 1727.
\hfill\qed

\bibliographystyle{plain} 
\bibliography{bibliofileCSB}

\end{document}